\documentclass[12pt]{amsart}
\usepackage{graphicx}
\usepackage{amsmath, amscd, amssymb, amsthm}
\usepackage[subnum]{cases}
\usepackage{changepage}
\usepackage{xcolor}
\usepackage{marginnote}
\usepackage{hyperref}
\usepackage{cleveref}
\usepackage[all]{xy}
\usepackage{tabularx}
\usepackage{ltablex}
\usepackage{longtable}
\usepackage[T1]{fontenc}

\newtheorem{theorem}{Theorem}[section]
\newtheorem{lemma}[theorem]{Lemma}

\newtheorem{proposition}[theorem]{Proposition}
\newtheorem{corollary}[theorem]{Corollary}
\newtheorem{conjecture}[theorem]{Conjecture}

\theoremstyle{definition}
\newtheorem{definition}[theorem]{Definition}
\newtheorem{remark}[theorem]{Remark}
\numberwithin{equation}{section}
\newtheorem{example}[theorem]{Example}
\newtheorem{assumption}[theorem]{Assumption}
\newtheorem{setting}[theorem]{Setting}

\usepackage{fancyhdr}
\begin{document}

\normalfont

\title{Hodge-Iwasawa Theory I}
\author{Xin Tong}

\maketitle

\begin{abstract}
\rm In this paper, we are going to establish a simultaneous generalization of the relative Iwasawa theory proposed by Kedlaya-Pottharst and the relative $p$-adic Hodge theory after Kedlaya-Liu. We call this Hodge-Iwasawa theory in the sense that one could apply the theory to study noncommutative Iwasawa cohomology and noncommutative Iwasawa theories in families and meanwhile one could apply the theory to study the deformation theory of \'etale local systems or families of representations of fundamental groups or the equivariant constructible $p$-adic sheaves, with more sophisticated point of view coming from Kato, Fukaya-Kato. We follow closely the approach of Kedlaya-Liu to study the corresponding modules and sheaves over the corresponding deformed version of the period rings and period sheaves.

\end{abstract}

\newpage

\tableofcontents

\newpage

\section{Introduction}

\subsection{Hodge-Iwasawa Theory and Higher Dimensional Modeling of the Weil Conjectures}

This paper deals with some simultaneous generalization of the noncommutative Iwasawa theory after Kedlaya-Pottharst and the relative $p$-adic Hodge theory in the style of Kedlaya-Liu, with further sophisticated philosophy in mind dated back to Fukaya-Kato and Kato. To make the motivations and the applications more serious and clear to the readers, let us start from recalling what happens along the intersection of the two pictures above.

\begin{example} \mbox{\textbf{(After Kedlaya-Pottharst)}}
Let $K$	be a finite extension of $\mathbb{Q}_p$. In \cite{KP1}, the authors proposed that one could consider a family $\{L\}_L$ parametrized by perfectoid Galois fields of intrinsic descriptions of the category $\widetilde{\Phi\Gamma}_{K,A}$ (with the notations in \cite{KP1}) in terms of $\varphi$-modules over $\widetilde{\Pi}_{L,A}$ in the family with the further action of the groups $G(L/K)$ for each $L$ in the family. These actually for a fixed $L/K$  will give rise to some perfectoid $p$-adic Hodge theory immediately. Then they considered some geometrization which is just some application of the theory in \cite[Chapter 9]{KL15}, namely one first forgets the action of the group $G_{L/K}$ but replaces this by pro-\'etale topology over $\mathrm{Spa}(K,\mathcal{O}_K)$. Then one defines an object $\mathcal{M}$ in the category $\widetilde{\Phi\Gamma}_{K,A}$ which has no a priori relation to any tower $L/K$. One recovers the Iwasawa theory from the sheaf attached to the Iwasawa algebra attached to the group $G(L/K)$. This makes more sense if we know clearly the structure of $G(L/K)$. The memorization of the Iwasawa structure back from the Iwasawa deformation (instead of taking the limit throughout some rigidized towers) will be some key observation in our development. 
\end{example}

\indent On the other hand, we have the following example:

\begin{example}
Let $A$ be now an Artin algebra over $\mathbb{Q}_p$, in this situation one could actually consider the local systems $\underline{A}$ over some specific smooth proper rigid analytic space $X/\mathbb{Q}_p$. Therefore one could consider the situation where this deforms from some $\mathbb{Q}_p$-local systems. Then one should be able to relate constructible $\underline{A}$-local systems to the corresponding deformation of the corresponding $\mathbb{Q}_p$ representation of the corresponding \'etale fundamental groups. One could then as in \cite{KL15} and \cite{KL16} consider the corresponding Galois groups of affinoids in both rigid geometry, adic geometry or perfectoid geometry. Then in these situations, one could consider the corresponding deformation problem. Furthermore one could then consider more general rigid affinoids $A/\mathbb{Q}_p$ and consider the corresponding families of representations of the Galois groups mentioned above over $A$. In order to study these one could then construct the corresponding $A$-relative relative period rings perfect generalizing \cite{KL15} and \cite{KL16}. 	
\end{example}

\indent Therefore now we consider the combination of the two pictures and the corresponding simultaneous generalization of the pictures in the relative $p$-adic Hodge theory and the Iwasawa theory in the sense mentioned above. Recall that the toric chart in \cite{KL16} could admit some cyclotomic deformation with the corresponding $\Gamma$-action where $\Gamma$ will be just $\mathbb{Z}_p^{d}$ for some rank $d$. Then one could consider the corresponding picture in the Iwasawa theory to consider the corresponding $(\varphi,\Gamma)$-cohomology of the corresponding multivariate cyclotomic deformations. We will call this generalized Iwasawa cohomology and the generalized Iwasawa theory, since over the perfectoid covering of the base space over $\widehat{\mathbb{Q}_p(p^{1/p^{\infty}})}$, one has the corresponding relative de Rham period rings. Therefore one could consider the generalized Bloch-Kato dual exponential maps and so on.\\

\indent One should somehow regard that Iwasawa theory is mimicking the celebrated Weil conjectures, namely the Riemann Hypothesis for the arithmetic of the algebraic curves. The conjectures themselves actually do not have too much sort of motivic motivation a priori, for instance one could just ask the questions on the counting problems within the arithmetic of algebraic curves. However the conjectures to some extent are definitely admitting deep and fundamental motivic point of views. Since the issue is that one could define the corresponding $L$-functions in terms of some Hodge structure for algebraic varieties over finite fields, namely the Frobenius actions on the $\ell$-adic \'etale cohomologies. Also one could actually directly reinterpret the corresponding number theoretic properties in term of the corresponding cohomological properties, for instance the zeros of zeta functions and the functional equations, which is further adopted by Perrin-Riou in her celebrated work for instance. Iwasawa established the corresponding point of view successfully in the context of number fields, which is actually to some extent very fascinating and deep. Although the original pictures from Iwasawa admit more class group formulation and the corresponding motivation (namely more Galois theoretic or more \'etale), which is similar to the original motivation of Weil conjectures, but following later philosophy like from Kato everything could be formulated quite cohomologically, namely due to some deep relationship and similarization between the \'etale cohomology and Galois cohomology.\\

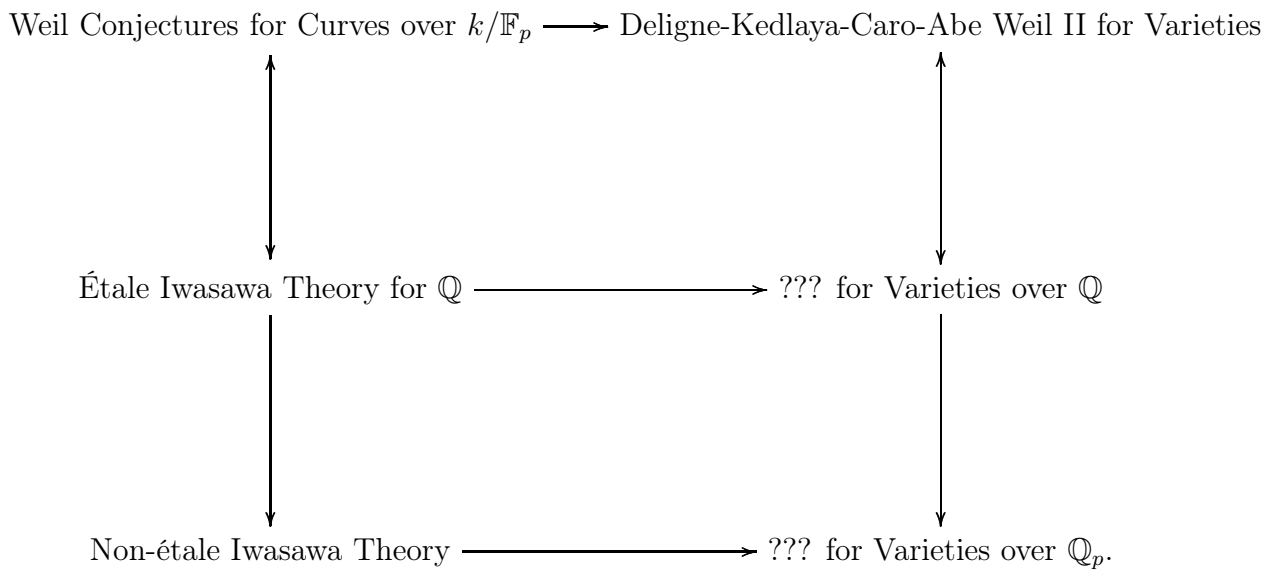
\[
\xymatrix@R+4.5pc@C+0pc{
\text{Weil Conjectures for Curves over $k/\mathbb{F}_p$}
\ar[r]\ar[r]\ar[r]\ar[d]\ar[d]\ar[d] &\text{Deligne-Kedlaya-Caro-Abe Weil II for Varieties} \ar[d]\ar[d]\ar[d]\\
\text{\'Etale Iwasawa Theory for $\mathbb{Q}$}
\ar[u]\ar[u]\ar[u]\ar[r]\ar[r]\ar[r]\ar[d]\ar[d]\ar[d] &\text{??? for Varieties over $\mathbb{Q}$} \ar[u]\ar[u]\ar[u]\ar[d]\ar[d]\ar[d].\\
\text{Non-\'etale Iwasawa Theory}
\ar[r]\ar[r]\ar[r] &\text{??? for Varieties over $\mathbb{Q}_p$}.
}
\]
\\

\indent As shown in the diagram above, one could see that actually Weil conjectures are quite general, in the sense that first they are actually above varieties over finite fields, which was originally proved by Deligne in terms of \'etale cohomology. By $p$-adic method, Kedlaya reproved the conjectures with actually more general point of view, since there one considers general smooth coefficients which is to say the isocrystals. The key issue is to establish some sort of finiteness of the rigid cohomology with such smooth non-\'etale coefficients. The method adopted by Kedlaya is through deep $p$-adic local monodromy theorem established in \cite{Ked1} which proves the affine curve cases directly, then through clever geometric method to reduce everything to affine spaces, and then to affine curves with induction. Actually the picture can be made more general in the sense that one could consider more general $p$-adic constructible coefficients, which was considered by Abe and etc, for instance the arithmetic $D$-modules.\\

\indent In the mixed-characteristic situation, Nakamura established some general Iwasawa theory for non-\'etale coefficients (by using Kedlaya-Liu's language these are pro-\'etale non-\'etale Hodge-Frobenius sheaves) over $X=\mathrm{Spa}(\mathbb{Q}_p,\mathcal{O}_{\mathbb{Q}_p})$. Through Kedlaya-Pottharst's point of view, namely the analytic cohomology of the cyclotomic deformation of pro-\'etale Hodge structures. We will give our understanding on the possible $???$ in the entries of the graph above, which is already represented at the beginning of this subsection. Replacing the cyclotomic towers or toric towers by pro-\'etale sites allows us to define the generalized Iwasawa cohomology. Therefore deformation of the Hodge sheaves over pro-\'etale sites will remember the Iwasawa theoretic consideration, which is observed as mentioned above by Kedlaya-Pottharst. \\


\indent  The first main topic exhibited in our main body of work is on the deformation over a reduced affinoid algebra $A$ the corresponding categories of vector bundles and the corresponding Frobenius modules (bundles) over the ring $\widetilde{\Pi}^\infty_{R,A}$ and $\widetilde{\Pi}_{R,A}$ (with the notations in \cref{corollary7.6} as below):

\begin{theorem}
The categories of $A$-relative quasicoherent finite locally free sheaves over the schematic relative Fargues-Fontaine curve $\mathrm{Proj}P_{R,A}$, the Frobenius modules over $\widetilde{\Pi}^\infty_{R,A}$, the Frobenius modules over $\widetilde{\Pi}_{R,A}$ and the Frobenius bundles over $\widetilde{\Pi}_{R,A}$ are equivalent. Moreover for those Fr\'echet-Stein distribution algebra $A_\infty(G)$ attached to a $p$-adic Lie group $G$ in the style of the Hodge-Iwasawa consideration assumed in \cref{setting3.8}, the corresponding statement holds as well (see \cref{section3.2}).\\
\end{theorem}

\indent This is actually considered in \cite{KP1} in the situation where $R$ is a perfectoid field. The main goal in our mind at this moment is essentially some belief (encoded in some possible application) that the deformations (both from some representation reason or Iwasawa-Tamagawa reason) share some rigid equivalence among the vector bundles over relative Fargues-Fontaine curves, $B$-pairs in families and finally the Frobenius-Hodge modules in relative sense. Of course if one focuses on the corresponding Galois representation context, then this belief will be somehow complicated while we believe the categories mentioned here are big enough to eliminate the difference from the deformations and the original absolute situation, without touching the \'etaleness. There are many contexts where such deformation on the level of $\varphi$-modules has been already proved to be important and convenient, such as \cite{Nak4} and \cite{KPX}. \cite{KPX} considered essentially the deformation theory of $\varphi$-modules by using the tools from Fr\'echet algebras. Note that $B$-pairs are generalizations of Galois representations from the point of view of the coefficients, Frobenius modules over ind-Fr\'echet sheaves are generalizations of Galois representations from the point of view of the slope theory and the corresponding purity. The relative deformation of the Hodge structures in such generality looks challenging, for instance the parallel story to Deligne-Laumon, \cite{AM}, \cite{G1} and \cite{G2} looks quite mysterious.\\

\indent We actually discuss in this paper the corresponding pseudocoherent objects as well in the sense of \cite{SGA6} and the corresponding pseudocoherent Frobenius modules studied in \cite{KL16}:

\begin{theorem}
The category of the pseudocoherent sheaves over the schematic Fargues-Fontaine curve $\mathrm{Proj}P_{R,A}$ and the category of the pseudocoherent modules over $\widetilde{\Pi}_{R,A}$ endowed with isomorphisms coming from the Frobenius pullbacks are equivalent.	
\end{theorem}

\indent One of our current \textit{main conjectures} (in the \'etale setting) in this article is quite general over a rigid analytic space $X$ (over $\mathbb{Q}_p$) separated and of finite type (in the category of adic spaces) with respect to an adic ring $T$ as in \cite{Wit1}, \cite{Wit2} and \cite[1.4]{Wit3}, and we consider the corresponding categories $\mathbb{D}_\mathrm{perf}(X_\sharp,T)~ (\sharp=\text{\'et},\text{pro\'et})$ of compatible families $(\mathcal{F}^\bullet_I)_{I\subset T}$ complexes perfect and dg-flat of $T/I$ constructible flat \'etale or pro-\'etale sheaves and we conjecture that we have the corresponding Waldhausen enrichment. To be more precise:

\begin{conjecture}
Assume now that $p$ is a unit in the ring $T$. Assume that we have $\mathbb{D}_\mathrm{perf}(X_\sharp,T)~ (\sharp=\text{\'et},\text{pro\'et})$ are Waldhausen categories and we have the corresponding continuous maps from the attached $K$-theory topological spaces to the one attached to $\mathbb{D}_\mathrm{perf}(T)$ induced by the direct image functor (which will make the space $X$ more restrictive). Then we have that the maps are null-homotopic in some canonical way. (See \cref{conjecture9.13})
\end{conjecture}

\indent Using relatively homotopical language, this could be directly related to the formulation adopted by Fukaya-Kato in Deligne's category when we focus on the corresponding \textit{spaces of 1 type} where the corresponding homotopy groups will vanish above the degree one. This has actually been considered extensively already in \cite{Wit1} in the context in terms of SGAIV and V. The corresponding categories we mentioned above looks very complicated to study in the general situation. Since our consideration is highly globalized due to the fact that we are considering the whole categories of all the corresponding interesting objects, instead of each single object. We have discussed some further development beyond the \'etale setting (in the $p$-adic setting) at the end of the paper. \\

\subsection{Why Relative $p$-adic Hodge theory?}

\indent Since our motivation comes from the deformation of representations of \'etale fundamental groups of higher dimensional spaces and the geometrization of Iwasawa theory, it is important to look inside the intrinsic Hodge structures through some coherent tools, which is why we need to consider the corresponding relative $p$-adic Hodge theory in some deformed setting. The main goal of the study around the fundamental groups as illustrated above actually is usually of the following aspects:

\begin{setting}\mbox{}\\
I. Try to understand the intrinsic geometric structures and its reconstruction effect on the varieties. \\
II. Try to understand the representations.	\\
III. Try to understand the deformation of representations.\\
IV. Try to understand the Iwasawa sheaves with coefficients in some equivariant setting (where the equivariance sometimes factors through the \'etale fundamental groups).\\
\end{setting}

\indent Essentially and especially in our setting, the generalization is actually more complicated than the usual situation, namely the representations are relative in the sense that it is varying over a variety which is somewhat as general as possible with some generalized coefficients. In the usual setting, one could regard everything as local systems in most of desired cases. Then using Scholze's pro-\'etale sites and pushing down one could define the corresponding period functors. In our setting one can definitely rely on the advantage of the original ideas of Kedlaya-Liu \cite{KL15} and \cite{KL16}, namely transforming the problem of representations of fundamental groups not to the local systems over pro-\'etale sites but to the modules over relative Robba rings over quasi pro-\'etale sites. Under general philosophy from Berger, everything on the $p$-adic Hodge module structure will be encoded in the module structure over relative Robba rings.

\subsection{Further Study}
In our current establishment, we have touched some foundational results along our consideration towards general Hodge-Iwasawa theory. Deformations of the corresponding Hodge structures in our consideration (namely the analytic module structures over Robba sheaves and rings) have beed shown by us to have some well-behaved properties. The aspect of deformation of such general Hodge structures will be definitely one of our long-term goal which should be parallel to those current existing comparison theorems and the existing various types of equivariant cohomological projects (for instance in the setting of \'etale cohomology and rigid cohomology). Over a rigid analytic space, the investigation of the corresponding families of the representations of the \'etale fundamental group seems to be a very hard problem for instance, especially when one would like to consider the equivariant setting, namely the corresponding equivariance coming from the quotient of the \'etale fundamental groups. We actually expect our insight should work in more general setting where one considers some general enough quotients here.\\

On the Iwasawa side, we expect that we can follow and generalize the original ideas in \cite{KP1} to geometrize the corresponding interesting towers in Tamagawa-Iwasawa theory. This is actually different from some standard construction in the original Iwasawa theory and $\Lambda$-adic Hodge theory. We expect this will give us a more uniform approach to recover many resulting Iwasawa theory of the general $p$-adic Hodge structures, for instance the corresponding exponential maps and the dual exponential maps, and the corresponding Iwasawa interpolation throughout some interesting character varieties, carrying some relativization. One another thing we would like to mention is that actually we have already motivated the corresponding study of the Iwasawa theory in the pseudocoherent setting beyond the vector bundle non-\'etale context.\\

Maybe a tour through the topological method in $p$-adic analysis should be possible in a long term consideration. Carrying some reasonable type of relativization and equivariance, one can apply higher categorical approach to study the corresponding relative $p$-adic Hodge theory in our context. We only here touched upon a piece of the story since our main goal here is Iwasawa deformation of general $p$-adic Hodge structure. But there is no reason to only restrict the consideration to this context. In other word the higher categorical investigation we limited here could motivate some further deeper understanding. Especially when we introduce the corresponding non-commutativity and higher-dimensional equivariance, where the context of the study of the rational $p$-adic Hodge theory looks quite challenging. On the other hand, actually the corresponding discussion above may also shed some light on the corresponding equivariant program in other contexts, such as the arithmetic $D$-modules. We would like to mention that actually the higher categorical aspects of $p$-adic Hodge theory are already encoded in some contexts in the literature, for instance the work \cite{BMS}.\\

\subsection{Notations}

\begin{center}
\begin{longtable}{p{7.8cm}p{8cm}}
Notation & Description \\
\hline
$p$ & A prime number.\\
$(R,R^+)$ & A perfect adic Banach algebra uniform over $\mathbb{F}_{p^h}$. \\
$A$ & A reduced affinoid algebra over $\mathbb{Q}_p$ in the rigid geometry after Tate. In general we consider the spectral norms on $A$ unless specified otherwise. \\
$T$ & An adic ring. Note that this is considered in the work for instance \cite{Wit1}, \cite{Wit2} and \cite[Section 1.4]{Wit3}.\\
$A_\infty(G)$ & Fr\'echet-Stein algebra attached to some nice $p$-adic group $G$ of type Lie. \\
$\Omega^\mathrm{int}_{R},	\Omega_{R}$ & Kedlaya-Liu's perfect period rings, over $R$. See \cite{KL15} and \cite{KL16}.\\
$\widetilde{\Pi}^{\mathrm{int},r}_{R},\widetilde{\Pi}^{\mathrm{bd},r}_{R},
\widetilde{\Pi}^{I}_{R},\widetilde{\Pi}^{+}_{R},\widetilde{\Pi}^{\infty}_{R},\widetilde{\Pi}_{R}$ & Kedlaya-Liu's perfect period rings, over $R$. See \cite{KL15} and \cite{KL16}.\\

$\Omega^\mathrm{int}_{R,A}$ & $A$-relative version of Kedlaya-Liu's perfect period rings, $\Omega^\mathrm{int}_{R,A}$.\\
$\Omega_{R,A}$ & $A$-relative version of Kedlaya-Liu's perfect period rings, $\Omega_{R,A}$.\\

$\widetilde{\Pi}^{I}_{R,A}$  & $A$-relative version of Kedlaya-Liu's relative perfect Robba rings.\\
$\widetilde{\Pi}^{r}_{R,A}$  & $A$-relative version of Kedlaya-Liu's relative perfect Robba rings.\\
$\widetilde{\Pi}^{\infty}_{R,A}$  & $A$-relative version of Kedlaya-Liu's relative perfect Robba rings.\\
$\widetilde{\Pi}^{+}_{R,A}$  & $A$-relative version of Kedlaya-Liu's relative perfect Robba rings.\\

$\widetilde{\Pi}^{\mathrm{int},r}_{R,A},\widetilde{\Pi}^{\mathrm{int}}_{R,A},
\widetilde{\Pi}^{\mathrm{int},+}_{R,A}$ & $A$-relative version of Kedlaya-Liu's relative perfect integral Robba rings.\\
$\widetilde{\Pi}^{\mathrm{bd},r}_{R,A},\widetilde{\Pi}^{\mathrm{bd}}_{R,A},
\widetilde{\Pi}^{\mathrm{bd},+}_{R,A}$ & $A$-relative version of Kedlaya-Liu's relative perfect bounded Robba rings.\\

$\underline{\underline{\Omega}}^\mathrm{int}_A,\underline{\underline{\Omega}}_A$ & $A$-relative version of Kedlaya-Liu's relative perfect sheaves over $\mathrm{Spa}(R,R^+)$.\\ 

$\underline{\underline{\widetilde{\Pi}}}_A^{\mathrm{int},r}, \underline{\underline{\widetilde{\Pi}}}_A^{\mathrm{int}}$ & $A$-relative version of Kedlaya-Liu's relative perfect sheaves over $\mathrm{Spa}(R,R^+)$. \\
$\underline{\underline{\widetilde{\Pi}}}_A^{\mathrm{bd},r},\underline{\underline{\widetilde{\Pi}}}_A^{\mathrm{bd}}$ & $A$-relative version of Kedlaya-Liu's relative perfect sheaves over $\mathrm{Spa}(R,R^+)$. \\

$\underline{\underline{\widetilde{\Pi}}}_A^{I},\underline{\underline{\widetilde{\Pi}}}_A^{\infty},\underline{\underline{\widetilde{\Pi}}}_A,\underline{\underline{\widetilde{\Pi}}}_A^{r}$ & $A$-relative version of Kedlaya-Liu's relative perfect sheaves over $\mathrm{Spa}(R,R^+)$.\\


$\Omega^\mathrm{int}_{R,A_\infty(G)}$ & $A_\infty(G)$-relative version of Kedlaya-Liu's perfect period rings, $\Omega^\mathrm{int}_{R,A_\infty(G)}$.\\
$\Omega_{R,A_\infty(G)}$ & $A_\infty(G)$-relative version of Kedlaya-Liu's perfect period rings, $\Omega_{R,A_\infty(G)}$.\\

$\widetilde{\Pi}^{I}_{R,A_\infty(G)}$  & $A_\infty(G)$-relative version of Kedlaya-Liu's relative perfect Robba rings.\\
$\widetilde{\Pi}^{r}_{R,A_\infty(G)}$  & $A_\infty(G)$-relative version of Kedlaya-Liu's relative perfect Robba rings.\\
$\widetilde{\Pi}^{\infty}_{R,A_\infty(G)}$  & $A_\infty(G)$-relative version of Kedlaya-Liu's relative perfect Robba rings.\\
$\widetilde{\Pi}^{+}_{R,A_\infty(G)}$  & $A_\infty(G)$-relative version of Kedlaya-Liu's relative perfect Robba rings.\\

$\widetilde{\Pi}^{\mathrm{int},r}_{R,A_\infty(G)},\widetilde{\Pi}^{\mathrm{int}}_{R,A_\infty(G)},
\widetilde{\Pi}^{\mathrm{int},+}_{R,A_\infty(G)}$ & $A_\infty(G)$-relative version of Kedlaya-Liu's relative perfect integral Robba rings.\\
$\widetilde{\Pi}^{\mathrm{bd},r}_{R,A_\infty(G)},\widetilde{\Pi}^{\mathrm{bd}}_{R,A_\infty(G)},
\widetilde{\Pi}^{\mathrm{bd},+}_{R,A_\infty(G)}$ & $A_\infty(G)$-relative version of Kedlaya-Liu's relative perfect bounded Robba rings.\\

$\underline{\underline{\Omega}}^\mathrm{int}_{A_\infty(G)},\underline{\underline{\Omega}}_{A_\infty(G)}$ & $A_\infty(G)$-relative version of Kedlaya-Liu's relative perfect sheaves over $\mathrm{Spa}(R,R^+)$.\\ 

$\underline{\underline{\widetilde{\Pi}}}_{A_\infty(G)}^{\mathrm{int},r}, \underline{\underline{\widetilde{\Pi}}}_{A_\infty(G)}^{\mathrm{int}}$ & $A_\infty(G)$-relative version of Kedlaya-Liu's relative perfect sheaves over $\mathrm{Spa}(R,R^+)$. \\
$\underline{\underline{\widetilde{\Pi}}}_{A_\infty(G)}^{\mathrm{bd},r},\underline{\underline{\widetilde{\Pi}}}_{A_\infty(G)}^{\mathrm{bd}}$ & $A_\infty(G)$-relative version of Kedlaya-Liu's relative perfect sheaves over $\mathrm{Spa}(R,R^+)$. \\

$\underline{\underline{\widetilde{\Pi}}}_{A_\infty(G)}^{I},\underline{\underline{\widetilde{\Pi}}}_{A_\infty(G)}^{\infty},\underline{\underline{\widetilde{\Pi}}}_{A_\infty(G)},\underline{\underline{\widetilde{\Pi}}}_{A_\infty(G)}^{r}$ & $A_\infty(G)$-relative version of Kedlaya-Liu's relative perfect sheaves over $\mathrm{Spa}(R,R^+)$.\\


$X$ & Locally $v$-ringed spaces or preadic spaces.\\

$X,X_{\text{\'et}},X_{\text{pro\'et}}$ & The corresponding analytic sites, \'etale sites, pro-\'etale sites of $X$.\\

\end{longtable}
\end{center}

\newpage

\section{Frobenius Modules over Ind-Fr\'echet Algebras}

\noindent Now we define the Frobenius modules in our setting, which are in some sense the generalized Hodge structures. Following \cite{KL15} and \cite{KL16} we know that actually the main tool in the study will be to find the links between the local systems or representations of fundamental groups with the Frobenius modules. Note that this is actually a common feature of $p$-adic cohomology theories or if you wish the $p$-adic weights after Deligne et al, which is to say in more detail: initially one has not the chance to see the intrinsic Hodge structures from the coefficient systems or the representations which is to say for instance the Frobenius structure, instead one can only have the chance to observe this after the application of Fontaine's functors. This is actually not the same as the situation in the initial Hodge theory and the $\ell$-adic weights theory in the sense again Deligne et al.\\

\indent So let us consider the following notations of Frobenius modules and Frobenius bundles, where the latter will be some collection of kind of sections with respect to each interval.

\begin{setting}
Following \cite[Chapter 5]{KL15}, in the notation of both \cite{KL16} and \cite{KL15} we will in our context consider the situation where $h=1$ in our initial consideration on the finite field $\mathbb{F}_{p^h}$. And note that we could have another parameter $a$ in our initial setting of the definitions, so we are going to consider $a$-th power Frobenius $\varphi^a$ in the following. This means that in the characteristic $0$ situation we will consider $\mathbb{Q}_{p^a}$ and in the characteristic $p>0$ we consider working over the finite field $\mathbb{F}_{p^a}$ for some chosen positive integer $a\geq 1$.
\end{setting}

\begin{definition}
For general $A$ which is a reduced affinoid algebra over $\mathbb{Q}_p$ with integral subring $\mathfrak{o}_A$, we define the following various $A$-relative version of the corresponding Period rings from \cite[Chapter 5]{KL15}:
\begin{align}
\widetilde{\Pi}_{R,A}:=\bigcup_{r}	\widetilde{\Pi}^{r}_{R,A},\widetilde{\Pi}^{r}_{R,A}:=	\widetilde{\Pi}^{r}_{R}\widehat{\otimes}_{\mathbb{Q}_p}A,
\end{align}
\begin{align}
\widetilde{\Pi}^{I}_{R,A}:=	\widetilde{\Pi}^{I}_{R}\widehat{\otimes}_{\mathbb{Q}_p}A,\widetilde{\Pi}^\infty_{R,A}:=	\widetilde{\Pi}^\infty_{R}\widehat{\otimes}_{\mathbb{Q}_p}A.
\end{align}
\begin{align}
\widetilde{\Pi}^{\mathrm{bd}}_{R,A}:=\bigcup_{r}	\widetilde{\Pi}^{\mathrm{bd},r}_{R}\widehat{\otimes}_{\mathbb{Q}_p}A,\widetilde{\Pi}^{\mathrm{bd},r}_{R,A}:=	\widetilde{\Pi}^{\mathrm{bd},r}_{R}\widehat{\otimes}_{\mathbb{Q}_p}A,
\end{align}
\begin{align}
\widetilde{\Pi}^{\mathrm{int}}_{R,A}:=\bigcup_{r}	\widetilde{\Pi}^{\mathrm{int},r}_{R}\widehat{\otimes}_{\mathbb{Z}_p}\mathfrak{o}_A,\widetilde{\Pi}^{\mathrm{int},r}_{R,A}:=	\widetilde{\Pi}^{\mathrm{int},r}_{R}\widehat{\otimes}_{\mathbb{Z}_p}\mathfrak{o}_A.
\end{align}

\end{definition}

\begin{remark}
One can also consider more deformed version of the period rings as in \cite{KL15} in some natural way which we do not explicitly write.
\end{remark}

\begin{example}
To be more concrete, in the situation where $A$ is $\mathbb{Q}_p\left<T_1,...,T_d\right>$ for some $d$, we can consider the following more explicit definitions. First form the ring of Witt vectors $W(R)$, then we consider the algebra $W(R)_A$ which is defined by taking the corresponding completed tensor product whose elements have the form of:
\begin{displaymath}
\sum_{n\geq 0,i_1\geq 0,...,i_n \geq 0}p^n[\overline{x}_{n,i_1,...,i_d}]T_1^{i_1}T_2^{i_2}...T_d^{i_d}.
\end{displaymath}
Over $W(R)_A$ we have the Gauss norm $\left\|.\right\|_{\alpha^r,A}$ for each $r>0$ which is defined by 
\begin{displaymath}
\left\|.\right\|_{\alpha^r,A}(\sum_{n\geq 0,i_1\geq 0,...,i_n \geq 0}p^n[\overline{x}_{n,i_1,...,i_d}]T_1^{i_1}T_2^{i_2}...T_d^{i_d}):=\sup_{n\geq 0,i_1\geq 0,...,i_n \geq 0}\{p^{-n}\alpha(\overline{x}_{n,i_1,...,i_d})^r\}.
\end{displaymath}
Then we define $\widetilde{\Pi}_{R,A}^{\mathrm{int},r}$ to be the subring of $W(R)_A$ consisting of all the elements satisfying:
\begin{displaymath}
\lim_{n,i_1,...,i_d}\left\|.\right\|_{\alpha^r,A}(p^{-n}\alpha(\overline{x}_{n,i_1,...,i_d})^r)=0.	
\end{displaymath}
And we put $\widetilde{\Pi}_{R,A}^{\mathrm{int}}$ to be the union of $\widetilde{\Pi}_{R,A}^{\mathrm{int},r}$ for all $r>0$. We then define $\widetilde{\Pi}_{R,A}^{\mathrm{bd},r}$ to be $\widetilde{\Pi}_{R,A}^{\mathrm{int},r}[1/p]$, and furthermore define the ring $\widetilde{\Pi}_{R,A}^{\mathrm{bd}}$ to be the corresponding union throughout all $r>0$. Then we define the Robba ring $\widetilde{\Pi}^I_{R,A}$ for some interval $I$ as the Fr\'echet completion of the ring $W(R)_A[1/p]$ by using the family of norms $\left\|.\right\|_{\alpha^t,A}$ for all $t\in I$. Then one could take specific intervals $(0,r]$ or $(0,\infty)$ to define the corresponding Robba rings $\widetilde{\Pi}^r_{R,A}$ and $\widetilde{\Pi}^\infty_{R,A}$. Finally we can define the whole Robba ring $\widetilde{\Pi}_{R,A}$ as the union of all the $\widetilde{\Pi}^r_{R,A}$ throughout all $r>0$.

\end{example}

\begin{setting}
We need to specify the corresponding Frobenius in our setting, where we will consider the Frobenius action which is $A$-linear but acts trivially over the $A$-part in the products appeared as above in the definitions. 	
\end{setting}

\begin{proposition}
For $A$ a reduced affinoid algebra over $\mathbb{Q}_p$ as defined above, we have the corresponding equality:
\begin{displaymath}
\widetilde{\Pi}^{[s_{1},r_{1}]}_{R,A}\bigcap\widetilde{\Pi}^{[s_{2},r_{2}]}_{R,A}=\widetilde{\Pi}^{[s_{1},r_{2}]}_{R,A}.	
\end{displaymath}
Here we assume that $0<s_1\leq s_2\leq r_1\leq r_2$.	
\end{proposition}

\begin{proof}
This could be proved by using the strategy which mimicks \cite[Section 5.2]{KL15}, or one could consider any representation taking the form of $\sum_{i}r_i\otimes a_i$ with $r_i$ in the non-relative Robba rings involved and $a_i\in A$, then the result follows from the situation without considering the algebra $A$.	
\end{proof}

\begin{definition} \label{Def2.6}
Consider the $A$-relative ring of periods 
\begin{center}
$\triangle:\Omega_{R,A},\widetilde{\Pi}^\mathrm{bd}_{R,A},\widetilde{\Pi}^\mathrm{int}_{R,A}$, $\widetilde{\Pi}_{R,A},\widetilde{\Pi}^+_{R,A},\widetilde{\Pi}^\infty_{R,A}$. 
\end{center}
Then we are going to define the $A$-relative $\varphi^a$-module over $\triangle$ to be a finite locally free $\triangle$-module carrying semilinear Frobenius action of $\varphi^a$. Moreover now for $\triangle:\widetilde{\Pi}^\mathrm{bd}_{R,A},\widetilde{\Pi}^\mathrm{int}_{R,A},\widetilde{\Pi}_{R,A}$ we have for each radius $r>0$ a module $M_{r}$ over $\triangle^{r}$ descending the module $M$ and the isomorphism 
\begin{center}
$\varphi^*M_r  \overset{\sim}{\rightarrow}M_r\otimes_{\triangle^{r}}\triangle^{r/q}$.
\end{center}
We will call the module $M_{r}$ over $\triangle^{r}$ a model of $M$. One can also define any $A$-relative Frobenius module to be one as above coming from some base change of some model $M_r$.
\end{definition}

\indent Then one has the corresponding bundles in our setting as in the following generalizing the situation in \cite[Section 6.1]{KL15}.

\begin{definition}
Consider the period ring $\triangle:=\widetilde{\Pi}_{R,A}$. We define a $\varphi^a$-module $M_{I}$ over $\triangle_{I}$ for some closed interval $I\subset (0,\infty)$ having the form of $[s,r]$ with $0<s\leq r/q$ to be a finite locally free module over $\triangle^{I}$ carrying the semilinear Frobenius and the isomorphism 
\begin{displaymath}
\varphi^*M_{I}\otimes_{\triangle^{[s/q,r/q]}}\triangle^{[s,r/q]}\overset{\sim}{\rightarrow}M_{I}\otimes_{\triangle^{[s,r]}}\triangle^{[s,r/q]}.
\end{displaymath}

Then under this restriction on the interval we define a $\varphi^a$-bundle to be a collection $\{M_{I}\}_{I}$ of finite locally free modules $M_{I}$ over each $\triangle^{I}$ such that for each pair of intervals such that $I\subset I'$ we have an isomorphism $\psi_{I',I}:M_{I'}\otimes \triangle^{I}\overset{\sim}{\rightarrow} M_{I}$ and for a triple of intervals such that $I\subset I'\subset I''$ we have the obvious cocycle condition $\psi_{I',I}\circ\psi_{I'',I'}=\psi_{I'',I}$. Note that we have only finished the definitions under the restriction on the intervals. Now as in \cite[Section 6.1]{KL15} we extend this to any more general interval $I$ by taking any pair $I\subset I'$ of  intervals and extracting some module $M_{I'}$ where $I'$ satisfies the restriction in our previous definition. Then also we extend the definition of $\varphi^a$-bundles to all the general intervals in our setting. Finally we define global section of a $\varphi^a$-bundle to be a collection of elements having the general form $\{v_{I}\}_{I}$ such that for each pair $I\subset I'$ we have that the image of $v_{I'}$ under the image of $\psi_{I',I}$ is $v_{I}$.
\end{definition}

\indent Now as in the usual situation it is natural to now consider the finiteness issue as in the following:

\begin{proposition}\mbox{\bf (After Kedlaya-Liu \cite[Lemma 6.1.4]{KL15})} \label{proposition2.8}
Consider now an arbitrary $\varphi^a$-bundle $\{M_{I}\}$ over $\widetilde{\Pi}_{R,A}$. Suppose now that for any interval $I\subset (0,\infty)$  we have that there exists a finite number generating elements $\mathbf{e}_1,...,\mathbf{e}_n$ for some uniform $n\in \mathbb{Z}$ for the finitely locally free module $M_{I}$. Then this set of elements could be promoted to be a generating set of the global section as a module but now over $\widetilde{\Pi}^\infty_{R,A}$.
\end{proposition}

\begin{proof}
This is generalization of the corresponding results in \cite[Lemma 6.1.4]{KL15} in the usual Hodge structures. Basically the main issue in the proof is to exhibit the desired finitely generatedness by some approximating process as in the usual situation. As in the usual situation we consider first the morphism $M_{[p^{-\ell},p^\ell]}\rightarrow (\widetilde{\Pi}^{[p^{-\ell},p^\ell]}_{R,A})^n$ by using the corresponding basis as in the assumption whose composite with the inverse $(\widetilde{\Pi}^{[p^{-\ell},p^\ell]}_{R,A})^n\rightarrow M_{[p^{-\ell},p^\ell]}$ is the identity map. And one could then bound the gap between the subspace norm and the quotient norm of $M_{[p^{-\ell},p^\ell]^n}$ by some constant $e_\ell$ as in \cite[Lemma 6.1.4]{KL15}. Now let $\mathbf{v}\in M$. Then the point is to extract for each $i=1,...,n$ and $\ell=0,1,...$ the following element $B_{i,\ell}\in \widetilde{\Pi}^{[p^{-\ell},p^\ell]}_{R,A}$ and $A_{i,\ell}\in \widetilde{\Pi}^\infty_{R,A}$ by the following fashion following \cite[Lemma 6.1.4]{KL15}, first for any $j<\ell$ we choose the corresponding $B_{i,j},j<\ell$ such that $\mathbf{v}-\sum_{i=1}^n\sum_{j\leq \ell}B_{i,j}\mathbf{e}_{i}=\sum_{i=1}A_{i,\ell}\mathbf{e}_i$ after which we choose $B_{i,\ell}$ to ensure that $\left\|.\right\|_{\alpha^{t},A}(B_{i,\ell}-A_{i,\ell})\leq p^{-1}e_\ell^{-1}\left\|.\right\|_{\alpha^{t},A}(A_{i,\ell})$ for each $i=1,...,n$ with $t \in [p^{-\ell},p^\ell]$. Then one could then finish as in \cite[Lemma 6.1.4]{KL15} since by the suitable choices we will have desired convergence which gives rise to a desired expression of $\mathbf{v}$ in terms of the basis given.

\end{proof}

\begin{lemma}
Consider the corresponding projection map from a $\varphi^a$-modules $M$ over $\widetilde{\Pi}_{R,A}$ to the corresponding module $M_{I}$ over $\widetilde{\Pi}_{R,A}^{I}$ for some interval $I$. Then we have that this operation is a tensor equivalence. 	
\end{lemma}

\begin{proof}
See \cite[Lemma 6.1.5]{KL15}. 
\end{proof}

\indent Then we study the Frobenius invariance in our setting which generalizes the absolute situation where we do not have any deformation at all. Here we are going to use the notation $M(n)$ to denote the corresponding twist of $M$ where the twist is defined to be the one where $\varphi$ acts by $p^{-n}$.

\begin{proposition}\label{prop2.9} \mbox{\bf (After Kedlaya-Liu \cite[Proposition 6.2.2]{KL15})} 
Now we consider a $\varphi^a$-bundle $M$ over $\widetilde{\Pi}_{R,A}$. Then we have that there exists some integer $N\geq 1$ such that for $n\geq N$ we have that $\varphi^a-1:M_{[s,rq]}(n)\rightarrow M_{[s,r]}(n)$ is surjective. One could also have the chance to take the integer to be $1$ if the module could be derived from some module defined over $\widetilde{\Pi}_{R,A}^\mathrm{int}$. 	
\end{proposition}

\begin{proof}
Following \cite[Proposition 6.2.2]{KL15}, we consider the quotient $r/s\leq q^{1/2}$. Then by considering the Frobenius actions one could further assume that $r\in [1,q]$. Then we will use the notation $A_{i,j}$ and $B_{i,j}$ to denote the corresponding matrices for the actions of $\varphi^{-a}$ and $\varphi^a$ respectively. Here $A_{i,j}$ has entries in $\widetilde{\Pi}^{[s,rq]}_{R,A}$ while the latter $B_{i,j}$ has entries in $\widetilde{\Pi}^{[s/q,r]}_{R,A}$. Then we put:
\begin{align}
c_1=\sup\{\left\|.\right\|_{\alpha^{t},A}(A_{i,j})|t\in [s,rq]\},\\
c_2=\sup\{\left\|.\right\|_{\alpha^{t},A}(B_{i,j})|t\in [s/q,r]\}.	
\end{align}
The goal is to extract from any element $\mathbf{w}$ in the target module a preimage $\mathbf{v}$ which amounts to solving the following equation:
\begin{displaymath}
\varphi^a\mathbf{v}-\mathbf{v}=\mathbf{w}.	
\end{displaymath}
Strictly speaking here we need some twisted version of this equation. By using our new $c_1$ and $c_2$ one chooses a suitable integer $N$ which shares the same property as the one in \cite[Proposition 6.2.2]{KL15}. Then we adapt the argument therein to our $A$-relative setting. We start by setting the coordinate of $\mathbf{w}$ to be $(x_1,...,x_r)$ where $r$ is the rank. First we take suitable constant $c>0$ such that we have for all $n\geq N$:
\begin{displaymath}
\varepsilon:=\sup\{p^{-n}c_1c^{-(q-1)r/q},p^{n}c_2c^{(q-1)s}\}<1.	
\end{displaymath}
Then as in \cite[Proposition 6.2.2]{KL15} one seeks suitable decomposition of $x_i$ taking the form of $x_i:=y_i+z_i$ where $y_i\in \widetilde{\Pi}_{R,A}^{[s/q,r]}$ and $z_i\in \widetilde{\Pi}_{R,A}^{[s,rq]}$ such that for the well-located $t$ as \cite[Lemma 5.2.9]{KL15} with the estimates therein through the relative version of the equalities in \cite[Lemma 5.2.9]{KL15}:
\begin{align}
\left\|.\right\|_{\alpha^t,A}(\varphi^{-a}(y_i))\leq c^{-(q-1)t/q}\left\|.\right\|_{\alpha^t,A}(x_i),\\
\left\|.\right\|_{\alpha^t,A}(\varphi^{a}(z_i))\leq c^{(q-1)t}\left\|.\right\|_{\alpha^t,A}(x_i),\\	
\end{align}
with the corresponding:
\begin{displaymath}
\left\|.\right\|_{\alpha^t,A}(y_i),	\left\|.\right\|_{\alpha^t,A}(z_i) \leq \left\|.\right\|_{\alpha^t,A}(x_i).
\end{displaymath}

(Note that this is actually $A$-relative version of the \cite[Lemma 5.2.9]{KL15}, where one considers the product norm $\left\|.\right\|_{\alpha^{t},A}$.) 

Then we set as in \cite[Proposition 6.2.2]{KL15} and compute:
\begin{align}
\sum_{i=1}^rx'_i\mathbf{f}_i &:=p^n\varphi^{-a}(\sum_{i=1}^ry_i\mathbf{f}_i)+p^{-n}\varphi^{a}(\sum_{i=1}^rz_i\mathbf{f}_i),	
\end{align}

which could also be deduced from the kind of definition by using the matrices $A_{i,j}$ and $B_{i,j}$. As in \cite[Proposition 6.2.2]{KL15} we further consider the corresponding coordinate vectors $y,z,x'$ as functions of $x$. And set the following iteration $x_{(\ell+1)}:=x'(x_{(\ell)})$ for each $\ell\geq 0$. To extract the desired $\mathbf{v}$ one just sets now:
\begin{displaymath}
\mathbf{v}:=\sum_{\ell\geq 0}-p^{n}\varphi^{-a}(\sum_{i=1}^ry_i(x_{(\ell)})\mathbf{f}_i)+\sum_{i=1}^rz_i(x_{(\ell)})\mathbf{f}_i.	
\end{displaymath}
As in the usual situation and by our construction above we have that this element will converge to an desired element in the domain module since we have 
\begin{center}
$\mathbf{v}-p^{-n}\varphi^{a}\mathbf{v}=\sum_{i=1}^rx_{(0),i}\mathbf{f}_i$.
\end{center} 
Finally one could then follow the strategy in \cite[Proposition 6.2.2]{KL15} to tackle the situation in more general setting by taking some suitable
\begin{center}
 $t=\max\{r/q^{1/2},s\}$ 
\end{center}
for those radii which do not satisfy the condition as above and extract suitable preimage $\mathbf{v}$ and then show that one can put this to be the desired element. To be more precise one considers the corresponding radius $t=\max\{r/q^{1/2},s\}$ where the corresponding radii $s,r$ may not satisfy the corresponding conditions as above. Then we have $t\geq r/q^{1/2}$ which implies that $r/t\leq q^{1/2}$ in this situation, which further implies (by the previous situation) one can extract a suitable element $\mathbf{v}_0$. To finish we still need to show that this element $\mathbf{v}_0$ is a desired preimage in our situation. Actually what we know in our situation is that the corresponding element $\mathbf{v}_0$ lives in the corresponding section $M_{[t,rq]}$, where this gives rise to the solution of the equation:
\begin{center}
$\mathbf{v}-p^{-n}\varphi^{a}\mathbf{v}=\sum_{i=1}^rx_{(0),i}\mathbf{f}_i$,
\end{center}
which further gives rise to the corresponding element in the corresponding section:
\begin{displaymath}
M_{[t/q,r]}.	
\end{displaymath}
Suppose now we a priori have the corresponding situation that:
\begin{displaymath}
t/q>s,	
\end{displaymath}
then we take the corresponding intersection in this situation with respect to the pair:
\begin{displaymath}
[t/q,r],[t,rq]	
\end{displaymath}
which gives rise to the fact that:
\begin{displaymath}
v_0\in M_{[t/q,rq]}	
\end{displaymath}
where if we have $t/q=s$ then we are done, otherwise we repeat the construction above again to produce the corresponding element gives rise to the situation:
\begin{align}
v_0\in M_{[t/q^2,rq]},\\
v_0\in M_{[t/q^3,rq]},\\
...,\\
v_0\in M_{[t/q^n,rq]},...,	
\end{align}
until there is some integer $N\geq 1$ such that we have $t/q^N\leq s$ then we are done.

\end{proof}


\begin{proposition} \label{proposition1} \mbox{\bf (After Kedlaya-Liu \cite[Proposition 6.2.4]{KL15})} 
Now again working over the ring $\widetilde{\Pi}_{R,A}$ and consider a $\varphi^a$-bundle which is denoted by $M$. Then we have that there exists some positive integer $N$ such that for all $n\geq N$ one could find finitely many $\varphi^a$-invariant global sections of $M(n)$ that generate $M$.	
\end{proposition}

\begin{proof}
We follow \cite[Proposition 6.2.4]{KL15}. In our generalized setting, first we take $s$ to be $rq^{-1/2}$. Then as in the situation of \cite[Proposition 6.2.4]{KL15} one chooses an element $\overline{\pi}$ whose norm is inside $(0,1)$. Then we consider the corresponding product which we will denote it by $\overline{\pi}$, and choose suitable rational number $z\in \mathbb{Z}[1/p]$ such that $[\overline{\pi}^{z}]$ gives rise to the value $c$ in the proof of the previous proposition. Then also as in the previous proposition one constructs (using $\varphi^a$ instead) the corresponding element $\mathbf{v}$ but in our situation we will denote this by $\mathbf{v}'_i$ for each $i$ suitable. To be more precise from the data above we put $\mathbf{v}'_i$ as the definition of $\mathbf{v}$ by using $x_{(0)}=0$ and $(y_{(0),j},z_{(0),j})=(-[\overline{\pi}^{s}],[\overline{\pi}^{s}])$ for $j=i$ and zero identically otherwise.\\
\indent Then we have the corresponding set of elements $\{\mathbf{v}'_i\}$ in $M_{[rq^{-1/2},rq]}$ consisting those elements which are invariant under the action of $\varphi^a$. We repeat the argument in the proof of the previous proposition as in the following. First one uses the notation $X_{i,j}$ to denote the corresponding matrix of the operator $\varphi^{-a}$ (certainly under some chosen set of generators with the specific rank $r$ in our situation) while one uses the notation $Y_{i,j}$ to denote the corresponding matrix of the operator $\varphi^a$, where from the corresponding matrix elements one could compute the following two numbers:
\begin{align}
b_1:=\mathrm{sup}_t\{\left\|.\right\|_{\alpha^t,A}(X_{i,j}),i,j\in \{1,2,...,r\},t\in [s,rq]\}\\
b_2:=\mathrm{sup}_t\{\left\|.\right\|_{\alpha^t,A}(Y_{i,j}),i,j\in \{1,2,...,r\},t\in [s/q,r]\}.	
\end{align}
Now by using our new scalars $b_1,b_2$ one could have the chance to extract some suitable $N\geq 1$ as in our previous proposition and in the corresponding result of \cite[Proposition 6.2.2]{KL15}. Then we consider the following construction through the corresponding iterated induction as in the following. The corresponding initial coordinates have been constructed in our situation as mentioned in the previous paragraph. Then we consider some general coordinate $(x_1,...,x_r)$ for $r$ the rank. Then as in \cite[Lemma 5.2.9]{KL15} for each $i=1,...,r$ one can make the following decomposition, namely putting $x_i=y_i+z_i$ with the following estimate with some chosen constant $c>0$:
\begin{align}
\left\|.\right\|_{\alpha^t,A}(\varphi^{-a}(y_i))\leq c^{-(q-1)t/q}\left\|.\right\|_{\alpha^t,A}(x_i),\\
\left\|.\right\|_{\alpha^t,A}(\varphi^{a}(z_i))\leq c^{(q-1)t}\left\|.\right\|_{\alpha^t,A}(x_i),\\	
\end{align}
and:
\begin{displaymath}
\left\|.\right\|_{\alpha^t,A}(y_i),	\left\|.\right\|_{\alpha^t,A}(z_i) \leq \left\|.\right\|_{\alpha^t,A}(x_i).
\end{displaymath}
This will be the key step in the iterating process, while the scalar $c$ will be chosen to guarantee that we have the following situation:
\begin{displaymath}
\delta:=\mathrm{max}\{p^{-n}b_1c^{-(q-1)r/q},p^{n}b_2c^{(q-1)s}\}<1.
\end{displaymath}
Then the corresponding iterating process will be further conducted in the following way, namely we consider the corresponding (well-defined due to our construction above):
\begin{displaymath}
\sum_{i=1}^rx_i'\mathbf{f}_i:=p^{n}\varphi^{-a}	\sum_{i=1}^ry_i\mathbf{f}_i+p^{-n}\varphi^{a}	\sum_{i=1}^rz_i\mathbf{f}_i.
\end{displaymath}
Note here that from each $x_i,i=1,...,r$ one can produce the corresponding elements $x_i',i=1,...,r$ which further produces the corresponding elements $y_i,z_i,i=1,...,r$, which implies that one can regard the corresponding elements $y_i,z_i,i=1,...,r$ as functions $y_i(x_1,...,x_r)$, $z_i(x_1,...,x_r),i=1,...,r$ of the original elements $x_i,i=1,...,r$. The whole iterating process is conducted through putting $x_{(\ell+1)}:=x'(x_{(\ell)})$ for each $\ell=0,1,...$, which produces the desired new elements by considering the following series:
\begin{displaymath}
\mathbf{v}'_i:=\sum_{\ell\geq 0}-p^{n}\varphi^{-a}	(\sum_{i=1}^ry_i(x_{(\ell)})\mathbf{f}_i) +\sum_{i=1}^rz_i(x_{(\ell)})\mathbf{f}_i.	
\end{displaymath}
As in our previous situation this process produces the desired elements which are invariant under the Frobenius operators in the sense that:
\begin{displaymath}
\mathbf{v}'_i-p^{-n}\varphi^a\mathbf{v}'_i=\sum_{i=0}^rx_{(0),i}\mathbf{f}_i.	
\end{displaymath}
Then we write $\mathbf{v}'_i$ as $[\overline{\pi}^{s}]\mathbf{f}_i+M\mathbf{f}$ with some matrix $\{M_{i,j}\}$ satisfying the condition that $\left\|.\right\|_{\alpha^{t},A}(M_{i,j})\leq \varepsilon \alpha(\overline{\pi})^{st}$ for each $t\in [rq^{-1/2},r]$. Therefore up to here we conclude that the matrix $[\overline{\pi}^s]+M$ will then consequently be invertible over the period ring $\widetilde{\Pi}_{R,A}^{[rq^{-1/2},r]}$ which implies the corresponding finite generating property for the interval $[rq^{-1},r]$ (by further repeating the corresponding construction above for the interval $[rq^{-1},rq^{-1/2}]$). Then to finish one could then iterate as in \cite[Proposition 6.2.4]{KL15} and apply the Frobenius action from $\varphi^a$.  Note that this will rely on the analog in our situation of \cite[Lemma 5.3.4]{KL15} in the $A$-relative setting (in our setting this could be resolved by considering the product Banach norms on the period rings we defined in this restricted setting, or by considering the corresponding norms on the period rings in the usual sense and $A$).
\end{proof}

\indent Then we consider the Frobenius invariants acting on an exact sequence, generalizing the corresponding result from \cite[Corollary 6.2.3]{KL15}:

\begin{corollary} \label{Corol2.12}
Suppose now we have an exact sequence of $\varphi^a$-modules over the corresponding period ring $\widetilde{\Pi}_{R,A}$ taking the form of $0\rightarrow M_\alpha\rightarrow M \rightarrow M_\beta\rightarrow 0$. Then there exists a positive integer $N$ such that for all $n\geq N$ we have the following exact sequence:
\[
\xymatrix@R+0pc@C+0pc{
0\ar[r]\ar[r]\ar[r] &M_\alpha(n)^{\varphi^a=1} \ar[r]\ar[r]\ar[r] &M(n)^{\varphi^a=1}
\ar[r]\ar[r]\ar[r] &M_\beta(n)^{\varphi^a=1}
\ar[r]\ar[r]\ar[r] &0
}.
\]
	
\end{corollary}

\begin{proof}
This is a direct consequence of the previous \cref{prop2.9}.	
\end{proof}

\indent The finite projective objects are considered in the results above, actually one can naturally extend the discussion above to the setting of pseudocoherent objects. Following \cite[4.6.9]{KL16} we consider the following results around the finite generated objects:

\begin{proposition} \label{propo2.13}
Let $M$ be a Frobenius module defined over $\widetilde{\Pi}_{R,A}$ (namely endowed with a semilinear action of the Frobenius operator) which is assumed to be finitely generated. Then we have that one can find then an integer $N\geq 0$ such that for all $n\geq N$, $H^0_{\varphi^a}(M(n))$ generates the module $M$ itself and we have that the $H^1_{\varphi^a}(M(n))$ vanishes. Also let $M$ be a Frobenius module defined over $\widetilde{\Pi}_{R,A}$ (namely endowed with a semilinear action of the Frobenius operator) which is assumed to be finite projective. Then we have that one can find then an integer $N\geq 0$ such that for all $n\geq N$, $H^0_{\varphi^a}(M(n))$ generates the module $M$ itself and we have that the $H^1_{\varphi^a}(M(n))$ vanishes. 
\end{proposition}

\begin{proof}
In this situation one can follow the proof of \cite[4.6.9]{KL16} and the strategy in the proof of the previous two propositions.	
\end{proof}

\newpage

\section{Contact with Schematic Relative Fargues-Fontaine Curves}

\subsection{Rigid Analytic Deformation of Schematic Relative Fargues-Fontaine curves}

\noindent In this section, we study some relationship between the vector bundles over Fargues-Fontaine curves and the generalized Frobenius modules defined in the previous section. This is relative version of the established results essentially in \cite{KL15}. The original picture could be dated back to the work of Fargues-Fontaine, where they established the classification of the Galois equivariant vector bundles over the so-called Fargues-Fontaine curves. For the convenience of the reader we first recall the definition of the scheme $\mathrm{Proj}P$ in the current picture.

\begin{setting}
Recall some algebraic geometry from \cite[Definition 6.3.1]{KL15} that first the ring $P$ is a graded commutative ring taking the form of $\bigoplus_{n\geq 0}P_n$ where each $P_{n}$ is defined to be a subring of $\widetilde{\Pi}^+_{R}$ 
which consists of all the elements which is invariant under the Frobenius $\varphi^a$ (up to scalars $p^n$ for each $n$). Then for each element $f$ of degree $d>0$ in $P_d$ we consider the local affine scheme $\mathrm{Spec}(P[1/f]_0)$, then we glue these to get the projective spectrum of $P$ as in \cite[Definition 6.3.1]{KL15} which is denoted by $\mathrm{Proj}P$. For more algebraic geometric discussion see \cite[Definition 6.3.1]{KL15}. 
\end{setting}

\indent Then we could generalize the situation in \cite[Definition 6.3.1]{KL15} to our $A$-relative situation.

\begin{setting}
In our situation, we consider the ring $\widetilde{\Pi}_{R,A}^+$ or the ring $\widetilde{\Pi}_{R,A}$ which is used to defined the corresponding graded ring as in the previous setting (namely, taking the corresponding $\varphi^a=p^n$ invariants for each $n\geq 0$), which will be denoted by $P_{R,A}$ or even just $P_A$.	To be slightly more precise $P_{R,A}$ is defined to be the corresponding direct sum $\bigoplus_{n\geq 0}P_{n,A}$.
\end{setting}

\begin{setting}
First we are going to use the notation $f$ to denote both the corresponding elements with specific degrees $d$ in $P_{d}$ and the corresponding elements in the tensor product of the local affine rings. And then we consider now the following invariance for any $M$, which is any $\varphi^a$-bundle over $\widetilde{\Pi}_{R,A}$. The invariance $M_{f}$ is now defined in our setting as $M[1/f]^{\varphi^a}$. To be more explicit we have $M_{f}=\bigcup_{n\geq 0}f^{-n}M(dn)^{\varphi^a}$. As in the usual situation one could regard this as a module over the tensor product $P_A[1/f]_0$, which then for any $I$ an interval induces the following map:

\begin{displaymath}
M_{f}\otimes_{P_A[1/f]_0}\widetilde{\Pi}^{I}_{R,A}[1/f]\rightarrow M_{f}\otimes_{\widetilde{\Pi}^{I}_{R,A}}\widetilde{\Pi}^{I}_{R,A}[1/f].	
\end{displaymath}
\end{setting}

\indent To study more, we first consider the following corollary after Kedlaya-Liu:

\begin{corollary}
Consider an arbitrary exact sequence of $\varphi^a$-modules taking the form of
\[
\xymatrix@R+0pc@C+0pc{
0\ar[r]\ar[r]\ar[r] &M_A \ar[r]\ar[r]\ar[r] &M
\ar[r]\ar[r]\ar[r] &M_B
\ar[r]\ar[r]\ar[r] &0.
}
\]
Then we have the derived corresponding exact sequence:
\[
\xymatrix@R+0pc@C+0pc{
0\ar[r]\ar[r]\ar[r] &M_{A,f} \ar[r]\ar[r]\ar[r] &M_{f}
\ar[r]\ar[r]\ar[r] &M_{B,f}
\ar[r]\ar[r]\ar[r] &0.
}
\]
			
\end{corollary}

\begin{proof}
Just consider the structure of $M_{f}$, then one could derive the results from our previous discussion.	
\end{proof}

\indent Then we have the following relative version of \cite[Corollary 6.3.4]{KL15}:

\begin{lemma}
Let $M$ be a Frobenius bundle as above, and take $M_1$ and $M_2$ two Frobenius subbundles whose summation is also a subobject. Then we actually have an equality $(M_1+M_2)_f\overset{\sim}{\rightarrow} M_{1,f}+M_{2,f}$.
\end{lemma}

\begin{proof}
This is basically an $A$-relative version of the corresponding result in \cite[Corollary 6.3.4]{KL15}.	
\end{proof}

\indent Then we consider the following key proposition for our further study:

\begin{proposition} \mbox{\bf (After Kedlaya-Liu \cite[Theorem 6.3.9]{KL15})} 
Now let $M$ again be a $\varphi^a$-bundle over $\widetilde{\Pi}_{R,A}$. Then we have that
\begin{displaymath}
M_{f}\otimes_{P_A[1/f]_0}\widetilde{\Pi}^{I}_{R,A}[1/f]\rightarrow M_{f}\otimes_{\widetilde{\Pi}^{I}_{R,A}}\widetilde{\Pi}^{I}_{R,A}[1/f]	
\end{displaymath}  
is a bijection, and $M_{f}$ is finite projective over $P_A[1/f]_0$.
\end{proposition}

\begin{proof}
We follow \cite[Theorem 6.3.9]{KL15} which gives basically from \cref{proposition1} that there exists $\mathbf{e}_1,...,\mathbf{e}_k$ which are the global sections of $M(dn)$ for some $n$ which generate $M$ itself. Then we have that considering $f^{-n}\mathbf{e}_1,...,f^{-n}\mathbf{e}_k$ could give rise to the corresponding surjectivity as in \cite[Theorem 6.3.9]{KL15}. Then we consider the corresponding presentation by using these generating elements from $\widetilde{\Pi}_{R,A}(-dn)^k(dn)$ mapping to $M(dn)$ which gives rise to a mapping which is surjective mapping from $M_1\rightarrow M$ (certainly here we have to consider the corresponding bundle in our context corresponding to the module $\widetilde{\Pi}_{R,A}(-dn)^k$ and then twist to establish the corresponding map as above), which furthermore gives rise to (by the corresponding observation parallel to \cite[Theorem 6.3.9]{KL15}) an exact sequence taking the form of
\[
\xymatrix@R+0pc@C+0pc{
0\ar[r]\ar[r]\ar[r] &M_2 \ar[r]\ar[r]\ar[r] &M_1
\ar[r]\ar[r]\ar[r] &M
\ar[r]\ar[r]\ar[r] &0
}.
\]
Then we consider the corresponding induced exact sequence taking the form of:
\[
\xymatrix@R+0pc@C+0pc{
0\ar[r]\ar[r]\ar[r] &M_{2,f} \ar[r]\ar[r]\ar[r] &M_{1,f}
\ar[r]\ar[r]\ar[r] &M_{f}
\ar[r]\ar[r]\ar[r] &0
},
\]
from the discussion above. Then to finish we look at the corresponding commutative diagram as in \cite[Theorem 6.3.9]{KL15} in the following:
\[\tiny
\xymatrix@R+5pc@C+0.7pc{
 &M_{2,f}\otimes_{P_A[1/f]_0}\widetilde{\Pi}^{I}_{R,A}[1/f] \ar[d]\ar[d]\ar[d] \ar[r]\ar[r]\ar[r] &M_{1,f}\otimes_{P_A[1/f]_0}\widetilde{\Pi}^{I}_{R,A}[1/f] \ar[d]\ar[d]\ar[d]
\ar[r]\ar[r]\ar[r] &M_{f}\otimes_{P_A[1/f]_0}\widetilde{\Pi}^{I}_{R,A}[1/f] \ar[d]\ar[d]\ar[d]
\ar[r]\ar[r]\ar[r] &0\\
0\ar[r]\ar[r]\ar[r] &M_{2,f}\otimes_{\widetilde{\Pi}^{I}_{R,A}}\widetilde{\Pi}^{I}_{R,A}[1/f]	 \ar[r]\ar[r]\ar[r] &M_{1,f}\otimes_{\widetilde{\Pi}^{I}_{R,A}}\widetilde{\Pi}^{I}_{R,A}[1/f]	
\ar[r]\ar[r]\ar[r] &M_{f}\otimes_{\widetilde{\Pi}^{I}_{R,A}}\widetilde{\Pi}^{I}_{R,A}[1/f]	
\ar[r]\ar[r]\ar[r] &0
},
\]
which finishes the proof on the bijectivity as in the absolute situation in \cite[Theorem 6.3.9]{KL15}. The finite projectiveness could be prove exactly by the same fashion in \cite[Theorem 6.3.9]{KL15} (also one can similarly derive the corresponding finitely-presentedness). To be more precise one can look at the exact sequence constructed above:
\[
\xymatrix@R+0pc@C+0pc{
0\ar[r]\ar[r]\ar[r] &M_2 \ar[r]\ar[r]\ar[r] &M_1
\ar[r]\ar[r]\ar[r] &M
\ar[r]\ar[r]\ar[r] &0
}.
\]
Then by taking suitable higher degree twist in our situation (by considering the corresponding vanishing of the cohomology $\mathrm{H}^1(M^\vee\otimes M_2(dN))$) one can get the following push-out diagram just as in \cite[Theorem 6.3.9]{KL15}:
\[
\xymatrix@R+2pc@C+2pc{
 M_2 \ar[d]\ar[d]\ar[d] \ar[r]\ar[r]\ar[r] &M_1 \ar[d]\ar[d]\ar[d]
\ar[r]\ar[r]\ar[r] &M \ar[d]\ar[d]\ar[d]\\
f^{-N}M_2	 \ar[r]\ar[r]\ar[r] &M'_{1}
\ar[r]\ar[r]\ar[r] &M	
},
\]
which gives rise to that $M_{1,f}\overset{\sim}{\rightarrow} M'_{1,f}$. Then by our construction, since $M_{1,f}$ is free we are done.

\end{proof}

\indent We then have the following results after suitable localization.

\begin{corollary} \label{corollary7.6}
Pick the element $f$ as in the previous proposition, we have that the following four categories are equivalent: \\
I. The category of finite locally free $P_A[1/f]_0$-modules over the relative coordinate ring of the localization of the Fargues-Fontaine curve $\mathrm{Proj}P_A$;\\
II.The category of $A$-relative $\varphi^a$-modules over $\widetilde{\Pi}^\infty_{R,A}[1/f]$;\\
III. The category of $A$-relative $\varphi^a$-modules over $\widetilde{\Pi}_{R,A}[1/f]$;\\
IV. The category of $A$-relative $\varphi^a$-bundles over $\widetilde{\Pi}_{R,A}[1/f]$.	
\end{corollary}

\begin{proof}
This is essentially by iteratedly applying the base change functor and the previous proposition.	
\end{proof}



\begin{proposition} \mbox{\bf (After Kedlaya-Liu \cite[Theorem 6.3.12]{KL15})}     \label{corollary7.6}
We have that the following four categories are equivalent: \\
I. The category of quasicoherent finite locally free sheaves of $\mathcal{O}_{\mathrm{Proj}P_A}$-modules (namely the $A$-relative vector bundles) over the Fargues-Fontaine curve $\mathrm{Proj}P_A$;\\
II.The category of $A$-relative $\varphi^a$-modules over $\widetilde{\Pi}^\infty_{R,A}$;\\
III. The category of $A$-relative $\varphi^a$-modules over $\widetilde{\Pi}_{R,A}$;\\
IV. The category of $A$-relative $\varphi^a$-bundles over $\widetilde{\Pi}_{R,A}$.	
\end{proposition}

\begin{proof}

\indent Let us mention briefly the corresponding functors involved since essentially the process is as in the same fashion of the construction of \cite[Theorem 6.3.12]{KL15}. First from II to III, this is just the base change, and the III to IV is the functor which maps any $\varphi^a$-module to the associated $\varphi^a$-bundle. Then for the functor from I to II one considers the process which associates any vector bundle $V$ in the first category some finite locally free sheaf $\mathcal{V}$ over the corresponding space associated to $\widetilde{\Pi}^\infty_{R,A}$ by considering the localization and glueing through the consideration by $f$ mentioned above, for the glueing we apply the direct analog of \cite[Lemma 6.3.7]{KL15} by considering \cref{proposition2.8} and \cref{proposition1}. Then one takes the global section to get desired module in the second category. Then in our situation the previous proposition gives us the final equivalence after the well-established glueing process under the consideration of the further relativization coming from the algebra $A$.
\end{proof}

\indent One can then further discuss the corresponding results on the comparison of pseudocoherent objects (just as in \cite[Definition 4.4.4]{KL16}), as in \cite[Section 4.6]{KL16}. The results for the pseudocoherent objects are bit more complicated than the results established above in the context of just vector bundles. We first have the following analog of \cite[Corollary 4.6.10]{KL16}:

\begin{proposition}
Let $M_\alpha,M,M_\beta$ be three finitely generated Frobenius modules over the ring $\widetilde{\Pi}_{R,A}$ (namely endowed with a semilinear action from the Frobenius operator). And now we put the modules then in an exact sequence:
\[
\xymatrix@R+0pc@C+0pc{
0\ar[r]\ar[r]\ar[r] &M_\alpha \ar[r]\ar[r]\ar[r] &M
\ar[r]\ar[r]\ar[r] &M_\beta
\ar[r]\ar[r]\ar[r] &0.
}
\]
Then we have first the following exact sequence for sufficiently large integer $n\geq 0$:
\[
\xymatrix@R+0pc@C+0pc{
0\ar[r]\ar[r]\ar[r] &M_\alpha(n)^{\varphi^a} \ar[r]\ar[r]\ar[r] &M(n)^{\varphi^a}
\ar[r]\ar[r]\ar[r] &M_\beta(n)^{\varphi^a}
\ar[r]\ar[r]\ar[r] &0.
}
\]
Then we have for any element $f$ which is a $\varphi^a=p^\ell$ invariance, after forming the following module for each $M_*:=M_\alpha,M,M_\beta$:
\begin{displaymath}
M_{*,f}:=\bigcup_{n\in \mathbb{Z}}f^{-n}M_{*}(\ell n)^{\varphi^a},	
\end{displaymath}
the following exact sequence of Frobenius modules over $(\widetilde{\Pi}_{R,A}[1/f])^{\varphi^a}$:
\[
\xymatrix@R+0pc@C+0pc{
0\ar[r]\ar[r]\ar[r] &M_{\alpha,f} \ar[r]\ar[r]\ar[r] &M_f
\ar[r]\ar[r]\ar[r] &M_{\beta,f}
\ar[r]\ar[r]\ar[r] &0.
}
\]
Eventually we have the fact that if $M'$ is then a pseudocoherent Frobenius module over $\widetilde{\Pi}_{R,A}$ then the module $M_f$ is also a pseudocoherent module over $(\widetilde{\Pi}_{R,A}[1/f])^{\varphi^a}$.
\end{proposition}

\begin{proof}
Apply \cref{propo2.13} as in \cref{Corol2.12} to derive the first two consequences. As in \cite[Corollary 4.6.10]{KL16} one can further prove the last statement by taking the corresponding projective resolutions and by using the previous statement repeatedly. 
\end{proof}

\begin{proposition}\mbox{\bf (After Kedlaya-Liu \cite[Theorem 4.6.12]{KL16})} Taking pullbacks along the map from the scheme associated to $\widetilde{\Pi}^\infty_{R,A}$ (by the analog of \cite[Lemma 6.3.7]{KL15} by using \cref{proposition1} and \cref{proposition2.8}) to the (schematic) Fargues-Fontaine curve (as in \cite[Definition 4.6.11]{KL16}) gives rise to an exact functor which establishes the equivalence between the category of pseudocoherent sheaves over the Fargues-Fontaine curve and the category of pseudocoherent Frobenius modules over $\widetilde{\Pi}_{R,A}$ (with isomorphism by the Frobenius pullbacks). This equivalence respects the corresponding comparison on the sheaf cohomologies and the Frobenius cohomologies.
	
\end{proposition}

\begin{proof}
We follow the strategy of the proof of \cite[Theorem 4.6.12]{KL16} to prove this. First in our context we have that the above functor is actually just initially known to be right exact under the base change. Now suppose $V$ is a pseudocoherent sheaf over the schematic Fargues-Fontaine curve in our context. As in the beginning of the proof of \cite[Theorem 4.6.12]{KL16} one has the following exact sequence:
\[
\xymatrix@R+0pc@C+0pc{
0\ar[r]\ar[r]\ar[r] &V_1 \ar[r]\ar[r]\ar[r] &V_2
\ar[r]\ar[r]\ar[r] &V
\ar[r]\ar[r]\ar[r] &0,
}
\]
where $V_1$ is pseudocoherent and $V_2$ is assumed and set to be a vector bundle. Then we apply the pullback construction (as mentioned in the statement of the proposition) we have the corresponding (after taking the further base change) exact sequence of the pseudocoherent Frobenius modules over the ring $\widetilde{\Pi}_{R,A}$:
\[
\xymatrix@R+0pc@C+0pc{
W_1 \ar[r]\ar[r]\ar[r] &W_2
\ar[r]\ar[r]\ar[r] &W
\ar[r]\ar[r]\ar[r] &0,
}
\] 
where $W_2$ is finite projective and $W_1$ is finitely generated. We then have the following exact sequence:
\[
\xymatrix@R+0pc@C+0pc{
0 \ar[r]\ar[r]\ar[r] &\mathrm{*} \ar[r]\ar[r]\ar[r] &W_1 \ar[r]\ar[r]\ar[r] &W_2
\ar[r]\ar[r]\ar[r] &W
\ar[r]\ar[r]\ar[r] &0,
}
\]
where $*$ is the kernel of the map $W_1\rightarrow \mathrm{Kernel}(W_2\rightarrow W)$. Then choose some $f$ with some specific degree as what we did before, by the previous proposition we have the following exact sequence:
\[
\xymatrix@R+0pc@C+0pc{
0 \ar[r]\ar[r]\ar[r] &\mathrm{*}_f \ar[r]\ar[r]\ar[r] &W_{1,f} \ar[r]\ar[r]\ar[r] &W_{2,f}
\ar[r]\ar[r]\ar[r] &W_f
\ar[r]\ar[r]\ar[r] &0,
}
\]
then by taking the corresponding section we have the following commutative diagram:
\[
\xymatrix@R+5pc@C+1.4pc{
 &0 \ar[r]\ar[r]\ar[r] &V_1|_{\widetilde{\Pi}_{R,A}[1/f]^{\varphi^a}} \ar[r]\ar[r]\ar[r] \ar[d]\ar[d]\ar[d] &V_2|_{\widetilde{\Pi}_{R,A}[1/f]^{\varphi^a}} \ar[d]\ar[d]\ar[d]
\ar[r]\ar[r]\ar[r] &V|_{\widetilde{\Pi}_{R,A}[1/f]^{\varphi^a}} \ar[d]\ar[d]\ar[d]
\ar[r]\ar[r]\ar[r] &0 \\
0 \ar[r]\ar[r]\ar[r] &\mathrm{*}_f \ar[r]\ar[r]\ar[r] &W_{1,f} \ar[r]\ar[r]\ar[r] &W_{2,f}
\ar[r]\ar[r]\ar[r] &W_f
\ar[r]\ar[r]\ar[r] &0
},
\]
where by our above results on the comparison on the finite projective objects we have that the second vertical morphism is isomorphism which implies that the third vertical map is surjective. Then we repeat the corresponding argument and the construction by using $V_1$ as our $V$ we can also derive the fact that the first vertical map is also at this situation surjective. Then as in \cite[Theorem 4.6.12]{KL16} the five lemma implies that the third vertical map is also injective. Then we apply the same construction and argument to the situation where we use $V_1$ to be our $V$ we can derive the fact that the first vertical map is also injective. Then we have that all the vertical maps are then in this situation isomorphism, which further implies that $*_f$ is trivial, as in \cite[Theorem 4.6.12]{KL16} by using the previous proposition we have that $*$ is trivial. The functor send finite projective objects to the corresponding finite projective objects so then we have the pseudocoherent objects in the corresponding essential image. Then we have a well-defined exact functor in our situation where the quasi-inverse is just taking the corresponding Frobenius invariance. The functor from the modules over the Robba ring to the sheaves over the Fargues-Fontaine curve will have the corresponding composition with the quasi-inverse being equivalence. On the other hand, to show that the functor from the sheaves over the Fargues-Fontaine curve will have the corresponding composition with the quasi-inverse being equivalence will eventually reduce to the corresponding statement in the finite projective setting as what we did before.
\end{proof}

\subsection{Fr\'echet-Stein Deformation of Schematic Relative Fargues-Fontaine curves} \label{section3.2}

\indent For our purpose and motivation, we would like to consider the corresponding situation where $A$ is replaced by the distribution algebras attached to some specific $p$-adic Lie groups or some mixed-type algebra by taking the product of $A$ with these interesting distribution algebras. Let us recall the following construction:

\begin{setting} \label{setting3.8}
Let $G$ be a $p$-adic Lie group in the fashion picked in the next setting, then we will use the notation $\mathcal{O}_K[[G]]$ to denote the corresponding integral completed group algebra over some finite extension $K/\mathbb{Q}_p$. Then we use the notation $K[[G]]$ to denote the base change of the integral ring to $K$. From here we consider the following integral version of the distribution we are considering:
\begin{displaymath}
A_n:=\mathcal{O}_K[[G]][\mathfrak{m}^n/p]^\wedge_{(p)}[1/p]	
\end{displaymath}
then take the inverse limit we get the ring $A_\infty$, where $\mathfrak{m}$ is the Jacobson radical. However this is not Noetherian, which makes life complicated. It is expected that one should consider this ring in order to do the equivariant or even more general noncommutative Iwasawa theory and the corresponding Tamagawa conjectures, in stead of just the usual Iwasawa algebra. For instance in the situation where $G$ is just $\mathbb{Z}_p^\times$ and $K=\mathbb{Q}_p$ this is just the ring $\Pi^\infty(\Gamma_K)$ which is just the Robba ring corresponding to the radius $\infty$. Also see \cite{Zah1} where this is discussed in more detail.
\end{setting}

\begin{setting}
In our current context we are going to focus on those groups in the following form:
\begin{displaymath}
G=\mathbb{Z}_p^d,\mathbb{Z}_p^\times,\mathbb{Z}_p^\times \ltimes \mathbb{Z}_p^n.	
\end{displaymath}
In this situation we just define the corresponding period rings deformed over the distribution algebra $A_\infty(G)$:
\begin{align}
&\widetilde{\Omega}^\mathrm{int}_{R,A_\infty(G)}:=\varprojlim_{n\rightarrow \infty}\widetilde{\Omega}^\mathrm{int}_{R,A^\mathrm{int}_n(G)},\widetilde{\Omega}_{R,A_\infty(G)}:=\varprojlim_{n\rightarrow \infty}\widetilde{\Omega}_{R,A_n(G)},\\
&\widetilde{\Pi}^{\mathrm{int},r}_{R,A_\infty(G)}:=\varprojlim_{n\rightarrow \infty}\widetilde{\Pi}^{\mathrm{int},r}_{R,A^\mathrm{int}_n(G)},\widetilde{\Pi}^\mathrm{int}_{R,A_\infty(G)}:=\varprojlim_{n\rightarrow \infty}\widetilde{\Pi}^\mathrm{int}_{R,A^\mathrm{int}_n(G)},\\	
&\widetilde{\Pi}^{\mathrm{bd},r}_{R,A_\infty(G)}:=\varprojlim_{n\rightarrow \infty}\widetilde{\Pi}^{\mathrm{bd},r}_{R,A_n(G)},\widetilde{\Pi}^\mathrm{bd}_{R,A_\infty(G)}:=\varprojlim_{n\rightarrow \infty}\widetilde{\Pi}^\mathrm{bd}_{R,A_n(G)},\\
&\widetilde{\Pi}^I_{R,A_\infty(G)}:=\varprojlim_{n\rightarrow \infty}\widetilde{\Pi}^I_{R,A_n(G)},	
\end{align}
from which one can define furthermore the period rings 
\begin{align}
\widetilde{\Pi}^r_{R,A_\infty(G)},\widetilde{\Pi}^\infty_{R,A_\infty(G)},\widetilde{\Pi}_{R,A_\infty(G)}.	
\end{align}
Then we work over the corresponding finite locally free modules or bundles over the corresponding period rings defined above with the same definitions we used in \cref{Def2.6} to define those modules over the period rings associated to $A_\infty(G)$. We also endow the rings defined above with partial Frobenius where the action is induced form the Witt vector part, not the coefficient algebra part. We then define the corresponding families of $A_\infty(G)$-relative $\varphi^a$-module $M$ over the ring $\widetilde{\Pi}_{R,A_\infty(G)}^\infty$ or $\widetilde{\Pi}_{R,A_\infty(G)}$ to be a projective system $(M_n)_{n}$ where each $M_n$ for each $n$ is a corresponding $A_{n}(G)$-relative $\varphi^a$-modules over $\widetilde{\Pi}_{R,A_n(G)}^\infty$ or $\widetilde{\Pi}_{R,A_n(G)}$ such that we have the base change requirement for each $n$:
\begin{displaymath}
A_{n}(G)\otimes_{A_{n+1}(G)}M_{n+1}\overset{\sim}{\rightarrow}M_{n}.	
\end{displaymath}
Similarly we define the Frobenius bundles in this compatible way. And finally we define the corresponding global section of any families of the Frobenius modules over the corresponding period rings to be the corresponding projective limit taking the form of $\varprojlim_{n\rightarrow\infty}M_n$.
\end{setting}

\begin{remark}
One can safely extend the assumption on the group $G$ to more general setting such that the corresponding algebra $A_\infty(G)$ are inverse limit of reduced (commutative) affinoid algebras.	
\end{remark}

\begin{definition}
We will in this situation use the notation $P_{A_\infty(G)}$ to mean the ring $\varprojlim_n P_{A_n(G)}$, therefore in our situation we have the corresponding ind-scheme $\mathrm{Proj}P_{A_\infty(G)}$ at the infinite level. Therefore we define the corresponding vector bundles over $\mathrm{Proj}P_{A_\infty(G)}$ to be the corresponding quasi-coherent finite locally free sheaves over the infinite level scheme $\mathrm{Proj}P_{A_\infty(G)}$. We then define the families of vector bundles 
to be those families taking the form of $(M_n)_n$ where each $M_n$ is a quasicoherent finite locally free sheaves $M_n$ over each $\mathrm{Proj}P_{A_n(G)}$ in the compatible way as in the previous setting. In this situation we define the corresponding global section of a family of vector bundles to be $\varprojlim_{n\rightarrow \infty} M_n$.	
\end{definition}

\begin{remark}
Here we actually have different notions of the corresponding modules over the corresponding Stein style. Since as in the usual story of the Iwasawa theory we have to study the corresponding derived categories, $K$-group spectra, the corresponding determinant category and so on, so we need to be careful when we would like to use the corresponding notions defined above. Here the main subtle point in the consideration is that actually it is nontrivial if the global sections of the corresponding sheaves are finitely generated. 	
\end{remark}

\indent For the corresponding families of the modules and the bundles we have the following comparison theorem:

\begin{proposition}
We have that the following four categories are equivalent: \\
I. The category of families of quasicoherent finite locally free sheaves of $\mathcal{O}_{\mathrm{Proj}P_{A_\infty(G)}}$-modules (namely the $A_\infty(G)$-relative vector bundles) over the Fargues-Fontaine curve $\mathrm{Proj}P_{A_\infty(G)}$;\\
II.The category of families of $A_\infty(G)$-relative $\varphi^a$-modules over $\widetilde{\Pi}^\infty_{R,A_\infty(G)}$;\\
III. The category of families of $A_\infty(G)$-relative $\varphi^a$-modules over $\widetilde{\Pi}_{R,A_\infty(G)}$;\\
IV. The category of families of $A_\infty(G)$-relative $\varphi^a$-bundles over $\widetilde{\Pi}_{R,A_\infty(G)}$.	
\end{proposition}

\begin{proof}
This will be further application of our established comparison above for $A_n(G)$ for each $n$.
\end{proof}

\indent And furthermore one could then further consider the following generality:

\begin{corollary}
We have that the following four categories are equivalent: \\
I. The category of families of quasicoherent finite locally free sheaves of $\mathcal{O}_{\mathrm{Proj}P_{A_\infty(G)\widehat{\otimes}_{\mathbb{Q}_p} A}}$-modules (namely the $A_\infty(G)\widehat{\otimes}_{\mathbb{Q}_p} A$-relative vector bundles) over the Fargues-Fontaine curve $\mathrm{Proj}P_{A_\infty(G)\widehat{\otimes}_{\mathbb{Q}_p} A}$;\\
II.The category of families of $A_\infty(G)\widehat{\otimes}_{\mathbb{Q}_p} A$-relative $\varphi^a$-modules over $\widetilde{\Pi}^\infty_{R,A_\infty(G)\widehat{\otimes}_{\mathbb{Q}_p} A}$;\\
III. The category of families of $A_\infty(G)\widehat{\otimes}_{\mathbb{Q}_p} A$-relative $\varphi^a$-modules over $\widetilde{\Pi}_{R,A_\infty(G)\widehat{\otimes}_{\mathbb{Q}_p} A}$;\\
IV. The category of families of $A_\infty(G)\widehat{\otimes}_{\mathbb{Q}_p} A$-relative $\varphi^a$-bundles over $\widetilde{\Pi}_{R,A_\infty(G)\widehat{\otimes}_{\mathbb{Q}_p} A}$.
\end{corollary}

\indent Then we will study in some generality established below the cohomology which will generalize the situation in \cite[Proposition 6.3.19]{KL15}. In our setting actually one could consider more general $A$. The algebras we are considering in the relative setting will be again the general Iwasawa theoretic level, namely the Fr\'echet algebra $A\widehat{\otimes}A_\infty(G)$. One can study the corresponding complexes of the Frobenius modules, or one can study the corresponding complexes of the corresponding families of the Frobenius modules and bundles.

\begin{proposition}
Consider now a $\varphi^a$-module $M$ over $\widetilde{\Pi}^\infty_{R,A}$ (which is always assumed to be finite projective). Then we have the following statements in our situation which generalize the corresponding results in \cite[Proposition 6.3.19]{KL15}:\\
I. Now consider the following commutative diagram:
\[
\xymatrix@R+2pc@C+2pc{
0 \ar[r]\ar[r]\ar[r] &M \ar[d]\ar[d]\ar[d]
\ar[r]^{\varphi^a-1}\ar[r]\ar[r] &M \ar[d]\ar[d]\ar[d]
\ar[r]\ar[r]\ar[r] &0\\
0 \ar[r]\ar[r]\ar[r] &M_r
\ar[r]^{\varphi^a-1}\ar[r]\ar[r] &M_{r/q}
\ar[r]\ar[r]\ar[r] &0.
}
\]
Then we have that the vertical maps give rise to the quasi-isomorphism of the two complexes involved;\\
II. In our situation, the natural base change $M\rightarrow M\otimes_{\widetilde{\Pi}^\infty_{R,A}}\widetilde{\Pi}_{R,A}$ gives rise to a quasi-isomorphism;\\
III. Consider the following commutative diagram for some interval $[s,r]$:
\[
\xymatrix@R+2pc@C+2pc{
0 \ar[r]\ar[r]\ar[r] &M \ar[d]\ar[d]\ar[d]
\ar[r]^{\varphi^a-1}\ar[r]\ar[r] &M \ar[d]\ar[d]\ar[d]
\ar[r]\ar[r]\ar[r] &0\\
0 \ar[r]\ar[r]\ar[r] &M_{[s,r]}
\ar[r]^{\varphi^a-1}\ar[r]\ar[r] &M_{[s,r/q]}
\ar[r]\ar[r]\ar[r] &0.
}
\]
Then again we have that the vertical map is a quasi-isomorphism of the two complexes involved. Here the corresponding radii satisfy the condition that $0<s\leq r/q$.
\end{proposition}

\begin{proof}
Largely this follows the same idea as in \cite[Proposition 6.3.19]{KL15} as below. The resulting complexes in our situation are those in the derived category of modules over ring $A$. 
We only need to prove I and III, then taking the limit will give us the second statement. Now for the first statement, as in the situation of \cite[Proposition 6.3.19]{KL15} the injectivity of both the maps
\begin{displaymath}
H_{\varphi^a_{}}^0(M)\rightarrow H_{\varphi^a_{}}^0(M_{r}),
\end{displaymath}
\begin{displaymath}
H_{\varphi^a_{}}^0(M)\rightarrow H^0_{\varphi^a}(M_{I})	
\end{displaymath}
follows due to the property of the finite projectiveness. Then we consider the corresponding surjectivity. In this case, we just need to consider as in \cite[Proposition 6.3.19]{KL15} the equation:
\begin{displaymath}
v=\varphi^{-na}v	
\end{displaymath}
for each $n\geq 1$, for instance considering those $v\in H_{\varphi^a_{}}^0(M_{r})$ and considering this will give rise to some element in $M_{q^{n}r}$, then taking the bootstrap like this to derived the result, which is the same for the situation in III. Then for the first order cohomology, one could just follow the idea in \cite[Proposition 6.3.19]{KL15} with the help of the previous comparison to lift the corresponding extension. 
We do not repeat the argument again.
\end{proof}

\begin{remark}
Here we only consider the corresponding comparison carrying the deformation to $A$ instead of more general situation where the deformation is over simultaneously $A$ and $A_\infty(G)$. But one could easily have the idea on what should be established in more general setting.	
\end{remark}


%
%
%
%
%

\newpage

\section{Hodge-Iwasawa Sheaves}

\subsection{Constructible $p$-adic Iwasawa Sheaves}

\indent It is actually quite natural to consider the corresponding constructible $p$-adic sheaves in Iwasawa theory as those considered in \cite{Wit1} on the interpolation of $L$-functions after Grothendieck from \cite{SGA5}. On the other hand, in the situation where one considers the Weil conjectures, it is actually more natural to consider both constructible and non-\'etale objects as in all types of Weil-II cnsidered by Deligne \cite{DWe1} and \cite{DWe2}, Kedlaya \cite{KWe2}, Caro, Abe \cite{CAWe2} and etc. Therefore we actually would like to consider the commutative equivariant version of the Hodge sheaves with Hodge structures considered in \cite{KL15}. We will not consider the generality just as in the previous section here, but we will consider the generality instead more general than \cite{KP1}.

\begin{setting}
All the flat constructible $p$-adic sheaves (both the $\mathbb{Z}_p$ ones or the $\mathbb{Z}_p$-isogeny ones), and all the Frobenius modules over the period Hodge sheaves will be considered in terms of some Iwasawa deformations, with the deformation in the corresponding reduced affinoids $A$ over $\mathbb{Q}_p$, or some adic ring $T$ (we will consider the situation considered by Witte in his noncommutative projects \cite{Wit1} and \cite{Wit3}), or a general Fr\'echet-Stein algebra $A_\infty(G)$ attached to some nice group $G$. In this section we assume that each adic ring is defined over $\mathbb{Z}_p$ and each quotient $T/I$ by some open two-sided ideal $I$ is of order a power $p^n$ of $p$.
\end{setting}

\begin{remark}
The Fr\'echet-Stein deformation over a point is essentially the one proposed by Kedlaya-Pottharst \cite{KP1}, which means that actually it is very important to consider the corresponding $\mathrm{DfmLie}$ namely the $p$-adic Lie deformation of the corresponding relative $A$ Hodge-Frobenius sheaves over the corresponding pro-\'etale site of a pre-adic space $X$, which will be a natural generalization of the corresponding constructions proposed in \cite{KP1} and established in \cite{KL15}.
\end{remark}

\indent We now start from the spaces in our setting, namely the pre-adic spaces $X$ over $\mathbb{Q}_p$ in the sense of those considered in \cite{KL15}. Then we will consider the corresponding homotopical categories of $p$-adic sheaves over the sites $X_{\text{\'et}}$ and $X_{\text{pro\'et}}$. As the setting considered by Witte in \cite{Wit1} and \cite{Wit3}, we will also consider the families of complexes with finite coefficients parametrized by the corresponding set of radicals $\mathfrak{J}$ of an adic ring $T$. \\

\indent In the style of the setting in Grothendieck's SGA IV and V, Witte studies the corresponding Iwasawa theory in the Fukaya-Kato setting, which is actually noncommutative. This is really general since the corresponding coefficients of the sheaves are actually in the adic ring mentioned above. The corresponding interpolation happens in the corresponding algebraic $K$-groups associated to some Waldhausen category. Waldhausen categories are natural places where one could do algebraic $K$-theory through Waldhausen $S$-constructions. We work out some generalization to the setting in our situations. We will consider from the following generality:

\begin{definition} \mbox{\bf (After Witte, \cite{Wit1},\cite[Chapter 5]{Wit3})}
Let $\mathfrak{J}$ be the set of all the two-sided ideals open of the adic ring $T$. Let $\mathbb{D}_\mathrm{perf}(X,T)$ denote the following category. Each object $(\mathcal{M}^\bullet_I)_{I\in \mathfrak{J}}$ of the category is now a family (to be more precise an inverse system) of the corresponding perfect (which are quasi-isomorphic to those strictly-perfect ones) complexes of constructible flat \'etale sheaves over $X_{\text{\'et}}$ parametrized by the corresponding open two-sided ideals of the adic ring $T$, such that: I. For each member of such family $(\mathcal{M}^\bullet_I)_{I\in \mathfrak{J}}$, we have that the complex $\mathcal{M}^\bullet_I$ is now supposed to be dg-flat over the ring $T/I$ (degreewise flat and the tensor product with any acyclic complex will be again acyclic, note here that the corresponding tensor product is taken over the ring $T/I$ when we are talking about each individual complex in a single family); II. We have the corresponding transition map $\psi_{I,J}:\mathcal{M}^\bullet_I\rightarrow \mathcal{M}^\bullet_J$ for any two open two-sided ideals $I\subset J$ with the basic base change requirement:
\begin{displaymath}
\mathcal{M}^\bullet_I\otimes_{T/I}T/J\overset{\sim}{\longrightarrow} \mathcal{M}^\bullet_J.	
\end{displaymath}
All the modules over any quotient $T/I$ is assumed to be left $T/I$-modules, while note that then the tensor product will happen when we have another right $T/I$-module.
\end{definition}

\begin{remark}
The information of the homotopy type comes from the weak-equivalence which is defined to be the quasi-isomorphism and the cofibration which is defined to be the injection with kernel in the original category.	
\end{remark}

\indent Note that Fukaya-Kato originally considered Deligne's virtual categories (strictly speaking through more transparent approach) where they formulated their celebrated $\zeta$-isomorphism conjecture and the local $\varepsilon$-conjectures. Witte used Waldhausen construction to give more $K$-theoretic pictures, which could be regarded as some reinterpretation of the Fukaya-Kato's picture in the corresponding homotopical categories. We are somehow deeply inspired by this point of view. Concentrating everything in the zero-th degree we have the following definition:

\begin{definition}
We consider the subcategory $\mathbb{D}^\mathrm{const}_{\mathrm{perf}}(X,T)$ of $\mathbb{D}_{\mathrm{perf}}(X,T)$ which consists of all the families of complexes degenerated to zeroth degree of flat construcible sheaves over $X_\text{\'et}$, again parametrized by elements in $\mathfrak{J}$, which to be more precise means that the corresponding coefficients will live in the coefficient ring taking the form of $T/I$. For each element we will then use the corresponding notation $(\mathcal{F}^\bullet_I)_{I\in \mathfrak{J}}$ to denote each family of flat constructible \'etale sheaves.
\end{definition}

This subcategory plays the role exactly the same as those flat constructible $\ell$-adic sheaves considered by Witte's picture following Deligne's construction and equivariant Abe-Caro's arithmetic $D$-module point of views. If one is more Hodge theoretic, then one should degenerate again to define:

\begin{definition}
We will use the notation $\mathbb{D}^\mathrm{const,smooth}_{\mathrm{perf}}(X,T)$ to denote the subcategory of $\mathbb{D}^\mathrm{const}_{\mathrm{perf}}(X,T)$ consists of all the families of complexes degenerated to zeroth degree of \'etale local systems $(\mathbb{L}_I)_{I\in \mathfrak{J}}$ parametrized by the elements in $\mathfrak{J}$.
\end{definition}

\indent These will be generalization of our consideration in the previous section, but restricting to the integral $p$-adic Hodge theory.

\begin{definition}
We call each object $(\mathbb{L}_I)_{I\in \mathfrak{J}}$ in the corresponding category
\begin{displaymath}
\mathbb{D}^\mathrm{const,smooth}_{\mathrm{perf}}(X,T)
\end{displaymath}
a noncommutative \'etale local system.
\end{definition}

\indent Then we could use this point of view to generalize the corresponding Fontaine's style equivalence in the previous section over perfectoid spaces immediately after \cite[Section 8.5]{KL15}:

\begin{proposition}
We have the following categories are equivalent for $X=\mathrm{Spa}(R,R^+)$ where $R$ is a perfect adic Banach uniform algebra over $\mathbb{F}_p$ with associated perfectoid space $X'=\mathrm{Spa}(A,A^+)$ under the perfectoid correspondence:\\
I. The category $\mathbb{D}^\mathrm{const,smooth}_{\mathrm{perf}}(X,T)$ of the noncommutative \'etale local systems over $X$;\\
II. The category $\mathbb{D}^\mathrm{const,smooth}_{\mathrm{perf}}(X',T)$ of the noncommutative \'etale local systems over $X'$.\\
\indent Moreover, we have that there is a fully faithful embedding from the category $\mathbb{D}^\mathrm{const,smooth}_{\mathrm{perf}}(X,T)$ (or $\mathbb{D}^\mathrm{const,smooth}_{\mathrm{perf}}(X',T)$ respectively) of all families of noncommmutative \'etale local systems over $X$ or $X'$ into the corresponding category of the families $(M_I)_I$ of $\varphi$-modules over $\widetilde{\Pi}_R^\mathrm{int}$ with deformation to $T$ where each $\mathcal{M}_{I}$ is a $\varphi$-module over $\widetilde{\Pi}_R^\mathrm{int}\otimes_{\mathbb{Z}_p}(T/I)$ in the usual sense such that we have the transition map $\psi_{I,J}$ as above with the additional base change properties. \\
\end{proposition}

\begin{proof}
In this situation this is derived from the corresponding result for the categories without Iwasawa deformation as proved in \cite[Section 8.5]{KL15}.	
\end{proof}

\begin{definition}
We define the following key categories over the site $X_{\text{pro\'et}}$ just as in the way as above by replacing $X$ with $X_{\text{pro\'et}}$:
\begin{displaymath}
\mathbb{D}^\mathrm{const,smooth}_{\mathrm{perf}}(X_{\text{pro\'et}},T)\subset \mathbb{D}^\mathrm{const}_{\mathrm{perf}}(X_{\text{pro\'et}},T)\subset \mathbb{D}_{\mathrm{perf}}(X_{\text{pro\'et}},T).\\
\end{displaymath}	
\end{definition}

\indent Then we glue everything in the previous two propositions over the corresponding pro-\'etale sites. Let $Y\in X_{\text{pro\'et}}$ be a perfectoid subdomain in the basis of the neighbourhood of this site, as in \cite[Section 9.2]{KL15} taking the form of $U=\varprojlim_{n\rightarrow \infty} U_n$. Then we use the corresponding notations $\mathcal{O}_X$, $\mathcal{O}^+_X$, $\mathcal{O}^\circ_X$ to denote the corresponding sheaves evaluating over $U$ as in \cite[Section 9.2]{KL15} as below:
\begin{align}
\mathcal{O}_X(U):=\varinjlim_{n} \mathcal{O}_{U_n}(U_n)\\
\mathcal{O}^+_X(U):=\varinjlim_{n} \mathcal{O}^+_{U_n}(U_n)\\
\mathcal{O}^\circ_X(U):=\varinjlim_{n} \mathcal{O}^\circ_{U_n}(U_n).
\end{align}

Recall that from \cite[Section 9.2]{KL15} we have the corresponding completed sheaves which we use the notations $\widehat{\mathcal{O}}_X,\widehat{\mathcal{O}}_X^+$ to denote them, then under the perfectoid correspondence we have the corresponding version of the sheaves over the pro-\'etale sites. Recall that these are denoted in \cite[Section 9.2]{KL15} by $\overline{\mathcal{O}}_X$ and $\overline{\mathcal{O}}_X^+$ following Andreatta-Iovita, by taking suitable tilting under the evaluation on any perfectoid subdomain $U$ as above. Then we are now at the position to glue the previous proposition over $X_{\text{pro\'et}}$.

\begin{definition}
Following \cite[Section 9.3]{KL15}, we are going to use the notation $\widetilde{\mathbf{A}}_X\widehat{\otimes}_{\mathbb{Z}_p}T$ and the notation $\widetilde{\mathbf{A}}_X\otimes_{\mathbb{Z}_p}(T/I)$ to be the sheafifications of the corresponding presheaves over $X_{\text{pro\'et}}$ sending each $U\in X_{\text{pro\'et}}$ to
\begin{align} 
W(\overline{\mathcal{O}}_X(U))\widehat{\otimes}_{\mathbb{Z}_p}T,
W(\overline{\mathcal{O}}_X(U))\otimes_{\mathbb{Z}_p}(T/I),
\end{align}
for each element $I\in \mathfrak{J}$. And we use the notations:
\begin{displaymath}
\widetilde{\mathbf{A}}^\dagger_X\widehat{\otimes}_{\mathbb{Z}_p}T,\widetilde{\mathbf{A}}^\dagger_X\otimes_{\mathbb{Z}_p}(T/I)
\end{displaymath}
to denote the corresponding sheafification of the presheaves sending each element $U$ in $X_{\text{pro\'et}}$ to:
\begin{align} 
\widetilde{\Pi}^\mathrm{int}_{\overline{\mathcal{O}}_X(U)}\widehat{\otimes}_{\mathbb{Z}_p}T,
\widetilde{\Pi}^\mathrm{int}_{\overline{\mathcal{O}}_X(U)}\otimes_{\mathbb{Z}_p}(T/I).
\end{align}
\end{definition}

\indent Then we have the following generalization of our previous proposition to more general setting over the pro-\'etale sites, again immediately after \cite[Theorem 9.3.7]{KL15} and its proof.

\begin{proposition}
We have the following categories are equivalent for $X$ defined above (as in the previous proposition):\\
I. The category $\mathbb{D}^\mathrm{const,smooth}_{\mathrm{perf}}(X_{\text{pro\'et}},T)$ of the noncommutative \'etale local systems over $X_{\text{pro\'et}}$;\\
II. The category $\mathbb{D}^\mathrm{const,smooth}_{\mathrm{perf}}(X'_{\text{pro\'et}},T)$ of the noncommutative \'etale local systems over $X_{\text{pro\'et}}'$.\\
 \indent Moreover, we have that there is a fully faithful embedding from the category $\mathbb{D}^\mathrm{const,smooth}_{\mathrm{perf}}(X_{\text{pro\'et}},T)$ (or $\mathbb{D}^\mathrm{const,smooth}_{\mathrm{perf}}(X'_{\text{pro\'et}},T)$ respectively) of all families of noncommmutative pro-\'etale local systems over $X$ or $X'$ into the corresponding category of the families $(M_I)_I$ of $\varphi$-modules over $\widetilde{\mathbf{A}}^\dagger_X$ with deformation to $T$ where each $\mathcal{M}_{I}$ is a $\varphi$-module over $\widetilde{\mathbf{A}}^\dagger_X\otimes_{\mathbb{Z}_p}(T/I)$ in the usual sense such that we have the transition map $\psi_{I,J}$ as above with the additional base change properties. \\
\end{proposition}

\subsection{Noncommutative-Equivariant $K$-Theory in Waldhausen Categories}

\indent Witte used $K$-theory to have had formulated some conjectures related to the corresponding Fukaya-Kato's key conjectures in Deligne's virtual category. Note that we defined key categories in the previous subsection for our further study on the Hodge-Iwasawa theory in the integral setting:
\begin{displaymath}
\mathbb{D}^\mathrm{const,smooth}_{\mathrm{perf}}(X,T)\subset \mathbb{D}^\mathrm{const}_{\mathrm{perf}}(X,T)\subset \mathbb{D}_{\mathrm{perf}}(X,T).
\end{displaymath}

\indent Also we have the following pro-\'etale version of the corresponding categories:

\begin{displaymath}
\mathbb{D}^\mathrm{const,smooth}_{\mathrm{perf}}(X_{\text{pro\'et}},T)\subset \mathbb{D}^\mathrm{const}_{\mathrm{perf}}(X_{\text{pro\'et}},T)\subset \mathbb{D}_{\mathrm{perf}}(X_{\text{pro\'et}},T).
\end{displaymath}

In the corresponding $K$-theory it is actually natural to ask for instance whether we have the corresponding special isomorphisms or homotopies as considered by Fukaya-Kato and Witte in the category of rigid analytic spaces or more general Huber's adic spaces. For instance one can consider the corresponding category $\mathbb{D}_{\mathrm{perf}}(\mathrm{Spa}(\mathbb{Q}_p,\mathfrak{o}_{\mathbb{Q}_p}),T)$ over a point. We actually put our observation into the following conjecture which actually predicts that over rigid analytic spaces we have the corresponding well-defined Waldhausen theory where in the algebraic geometry this is conjectured and studied by Witte in \cite{Wit1}, \cite{Wit2} and \cite[Section 1.4]{Wit3}. First let us consider the following definition:

\begin{definition} \mbox{\bf (After Witte, \cite{Wit1},\cite[Chapter 5]{Wit3})}
Let $\mathfrak{J}$ be the set of all the two-sided ideals open of the adic ring $T$ where we assume that $p$ is a unit in $T$. Let $\mathbb{D}_\mathrm{perf}(X,T)$ denote the following category. Each object $(\mathcal{M}^\bullet_I)_{I\in \mathfrak{J}}$ of the category is now a family (to be more precise an inverse system) of the corresponding perfect (which are quasi-isomorphic to those strictly-perfect ones) complexes of constructible flat \'etale sheaves over $X_{\text{\'et}}$ parametrized by the corresponding open two-sided ideals of the adic ring $T$, such that: I. For each member of such family $(\mathcal{M}^\bullet_I)_{I\in \mathfrak{J}}$, we have that the complex $\mathcal{M}^\bullet_I$ is now supposed to be dg-flat over the ring $T/I$ (degreewise flat and the tensor product with any acyclic complex will be again acyclic, note here that the corresponding tensor product is taken over the ring $T/I$ when we are talking about each individual complex in a single family); II. We have the corresponding transition map $\psi_{I,J}:\mathcal{M}^\bullet_I\rightarrow \mathcal{M}^\bullet_J$ for any two open two-sided ideals $I\subset J$ with the basic base change requirement:
\begin{displaymath}
\mathcal{M}^\bullet_I\otimes_{T/I}T/J\overset{\sim}{\longrightarrow} \mathcal{M}^\bullet_J.	
\end{displaymath}
All the modules over any quotient $T/I$ is assumed to be left $T/I$-modules, while note that then the tensor product will happen when we have another right $T/I$-module. And we also have the corresponding categories in the pro-\'etale setting.
\end{definition}

\begin{remark}
Note that here we consider the corresponding situation where $p$ is a unit in the adic ring $T$ which is different from the previous section.	
\end{remark}

\begin{conjecture} \label{conjecture9.13} \mbox{\bf (After Witte)}
With the assumption in the previous definition, for $X$ a rigid analytic space over the $p$-adic field $\mathbb{Q}_p$ which is separated and of finite type (when considered as some adic space locally of finite type, see \cite{Huber1}) such that for $\sharp=\text{\'et},\text{pro\'et}$ the categories $\mathbb{D}_{\mathrm{perf}}(X_\sharp,T)$ could be endowed with the structure of Waldhausen categories, and we have there is a Waldhausen exact functor $R\Gamma(X_\sharp,.)$ (induced by the corresponding direct image functor as in \cite[Chapter 4-5]{Wit3}) as below for $\sharp=\text{\'et},\text{pro\'et}$:
\[
\xymatrix@R+0pc@C+3pc{
\mathbb{D}_{\mathrm{perf}}(X_\sharp,T)\ar[r]^{R\Gamma(X_\sharp,.)}\ar[r]\ar[r] &\mathbb{D}_{\mathrm{perf}}(T)
}
\]
which induces the corresponding the morphism $\mathbb{K}R\Gamma(X_\sharp,.)$ as:
\[
\xymatrix@R+0pc@C+3pc{
\mathbb{K}\mathbb{D}_{\mathrm{perf}}(X_\sharp,T)\ar[r]^{\mathbb{K}R\Gamma(X_\sharp,.)}\ar[r]\ar[r] & \mathbb{K}\mathbb{D}_{\mathrm{perf}}(T),
}
\]	
we conjecture that there is a homotopy between this morphism and the zero map:
\[
\xymatrix@R+0pc@C+3pc{
\mathbb{K}\mathbb{D}_{\mathrm{perf}}(X_\sharp,T)\ar[r]^{0}\ar[r]\ar[r] & \mathbb{K}\mathbb{D}_{\mathrm{perf}}(T).
}
\]
And we conjecture this will be compatible with the higher direct image functor $Rf_!$ between different spaces (see Huber's book \cite[Chapter 5]{Huber1}) and the changing coefficient morphism for a pair $T,T'$ of adic rings.
\end{conjecture}

\begin{remark}
We are actually not quite for sure how much higher categorical information one could see through the potential enrichment. But one might want to ask this even in the scheme situation. 
\end{remark}

\indent We would like to start from here (in the future consideration) to study the corresponding $K$-theoretic construction after the work of \cite{Wit1} and \cite{Wit3}. We would also like to start from in the future consideration the corresponding schematic version of the construction in Witte's thesis for instance in \cite{Wit3}, where \cite{Wit3} has already constructed the corresponding categories of \'etale sheaves and the corresponding functors and properties.

\begin{example}
Suppose now everything is just over the point associated to the point $\mathbb{Q}_p$, namely the space $\mathrm{Spa}(\mathbb{Q}_p,\mathfrak{o}_{\mathbb{Q}_p})$. Then the story collapses to the corresponding Galois cohomology and the corresponding Waldhausen category of the derived Galois representations, which is the Waldhausen categorical generalization of the usual Galois theoretic Tamagawa-Iwasawa theory. To be more precise we have first of all that the categories 
\begin{displaymath}
\mathbb{D}_{\mathrm{perf}}(\mathrm{Spa}(\mathbb{Q}_p,\mathfrak{o}_{\mathbb{Q}_p})_\sharp,T)	
\end{displaymath}
could be endowed with structure of Waldhausen categories as shown in \cite[Chapter 5, 6]{Wit3}. 
Moreover we have the following Waldhausen exact functor (see \cite[Chapter 5]{Wit3}):
\[
\xymatrix@R+0pc@C+0pc{
R\Gamma(\mathrm{Spa}(\mathbb{Q}_p,\mathfrak{o}_{\mathbb{Q}_p})_\sharp,.):\mathbb{D}_{\mathrm{perf}}(\mathrm{Spa}(\mathbb{Q}_p,\mathfrak{o}_{\mathbb{Q}_p})_\sharp,T)
\ar[r]\ar[r]\ar[r] &\mathbb{D}_{\mathrm{perf}}(T)
}
\]
which gives rise to the following map on the spectra:
\[
\xymatrix@R+0pc@C+0pc{
KR\Gamma(\mathrm{Spa}(\mathbb{Q}_p,\mathfrak{o}_{\mathbb{Q}_p})_\sharp,.):K\mathbb{D}_{\mathrm{perf}}(\mathrm{Spa}(\mathbb{Q}_p,\mathfrak{o}_{\mathbb{Q}_p})_\sharp,T)\ar[r]\ar[r]\ar[r] &K\mathbb{D}_{\mathrm{perf}}(T).
}
\]	

\end{example}

\begin{remark}
Recall that actually in the categories of the virtual objects, Nakamura obtains partial results on the existence of the corresponding $\varepsilon$-isomorphism for \'etale $(\varphi,\Gamma)$-modules, for more detail see \cite{Nak1} and \cite{Nak2}. We only expect that one can construct the corresponding homotopy for the Waldhausen categorical consideration. Therefore in the situation where $p$ is not a unit in $T$ we still conjecture that there should be such Waldhausen category and the corresponding induced maps on the $K$-theoretic spaces, and even the null-homotopy. Note that this is some generalization of the corresponding picture proposed in \cite{KP1} in the \'etale situation. Here we are considering derived \'etale $\varphi$-sheaves over the pro-\'etale sites in the corresponding Waldhausen categories. The situation of a point as above corresponding to $\mathbb{Q}_p$ is already interesting to some extent.	We will continue our discussion in this situation  below.
\end{remark}

\subsection{Hodge-Iwasawa Period Sheaves and Hodge-Iwasawa Vector Bundles}

\indent We consider here vector bundles over the corresponding adic version of the Fargues-Fontaine curves and consider the relationship between these and the ones we encountered in the last section. We are going to first recall the basics about the Fargues-Fontaine curves in the adic setting. We start from the following picture in the local situation, namely over some perfectoid ring $R$ which is some member in the basis of topology of an adic spaces in some well-established category.

\begin{setting}
We recall some adic geometry from \cite[Definition 8.7.4]{KL15} around the corresponding geometry of the adic version of the Fargues-Fontaine curve with respect to the algebra $R$ as in our discussion above, which is to be denoted as in \cite[Definition 8.7.4]{KL15} $\mathrm{FF}_R$. For the convenience of the readers we recall the basic construction of the curves on the level of the adic spaces. In this local setting the curve $\mathrm{FF}_R$ is defined to be the quotient of the corresponding adic space:
\begin{displaymath}
O_R:=\bigcup_{0<s<r}\mathrm{Spa}(\widetilde{\Pi}^{[s,r]}_R,\widetilde{\Pi}^{[s,r],+}_R).	
\end{displaymath}
See \cite[Definition 8.7.4]{KL15}. Then by the basic construction we could have the chance to find a link between the adic Fargues-Fontaine curve and our originally considered schematic Fargues-Fontaine curve. To be more precise, we consider the following map:
\begin{displaymath}
P_R\rightarrow \widetilde{\Pi}^I_R	
\end{displaymath}
which naturally induces a morphism in the category of locally ringed spaces:
\begin{displaymath}
\mathrm{FF}_R\rightarrow \mathrm{Proj}P_R	
\end{displaymath}
coming from the one attached to the space $O_R$. Then we could consider the Iwasawa deformation of the vector bundles over these different locally-ringed spaces in the large categories of locally-ringed spaces.
\end{setting}

%
%
%

\indent The corresponding deformed version of the adic Fargues-Fontaine curves could be defined in the same way. But note that here we need to require the corresponding deformed period rings are sheafy in the category of adic rings. This is in general not a trivial statement. For instance if we are in the situation where the affinoid algebra $A$ is a Tate algebra then we know that this is indeed the case. However in the full generality this is not that easy to see to be true. We now choose to work around this by just putting some assumption on the algebra $A$ which is already applicable enough in some interesting enough situations. Recall from \cite{KH} that when a ring is sousperfectoid then we should have the following nice result. Recall that (from \cite[Introduction]{KH}) a Huber ring $S$ is called sousperfectoid if it admits a split continuous morphism to a perfectoid Huber ring $P$. Note that this morphism is considered in the category of $S$-modules which are topological.

\begin{proposition} \mbox{\bf (Kedlaya-Hansen)}
The sousperfectoid rings are sheafy.	
\end{proposition}

\begin{proof}
See \cite[Corollary 7.4]{KH}.	
\end{proof}

\begin{example}
Sousperfectoid rings are not rare, for instance one can just take the corresponding perfectoid field $\mathbb{Q}_p(p^{p^{-\infty}})^\wedge$ then $\mathbb{Q}_p$ will be then sousperfectoid in this larger perfectoid ring which admits a splitting. Also the ring $\widetilde{\Pi}^I_{R}$ is then sousperfectoid with respect to $\widetilde{\Pi}^I_{R}\widehat{\otimes}_{\mathbb{Q}_p}\mathbb{Q}_p(p^{p^{-\infty}})^\wedge$. More general, smooth affinoid algebras are sousperfectoid.	
\end{example}

\begin{assumption}
From now on, we assume that $A$ is sousperfectoid.	
\end{assumption}

\begin{corollary}
For any closed interval $I$ we have that the ring $\widetilde{\Pi}^{[s,r]}_{R,A}$ is sheafy.	
\end{corollary}

\begin{remark} Base on the observation above, we actually make the following conjecture:
\begin{conjecture}
The ring $\widetilde{\Pi}^{[s,r]}_{R,B}$ is stably uniform for any reduced affinoid algebra $B$, where the corresponding observation is as in the following. 
\end{conjecture}
\noindent Note that a Tate affinoid algebra over a uniform Banach ring is uniform, so we have that the ring $\widetilde{\Pi}^{[s,r]}_{R,T_d}$ is uniform, but it is not the case that the rational localization of $\widetilde{\Pi}^{[s,r]}_{R,T_d}$ is the product of rational localizations of $\widetilde{\Pi}^{[s,r]}_{R}$ and $T_d$, which is to say that one does not have the chance to prove this directly along this. However since the ring $\widetilde{\Pi}^{[s,r]}_{R,T_d}$ is sousperfectoid,  
so we have that the corresponding ring $\widetilde{\Pi}^{[s,r]}_{R,T_d}$ in the situation of Tate algebra is stably uniform. Now for general $A$ we suspect this is also true.
\end{remark}

\begin{setting}
In the same fashion, in our situation we consider the corresponding space 
\begin{displaymath}
O_{R,A}:=\bigcup_{0<s<r}\mathrm{Spa}(\widetilde{\Pi}^{[s,r]}_{R,A},\widetilde{\Pi}^{[s,r],+}_{R,A})	
\end{displaymath}	
which naturally gives rise to the corresponding quotient $\mathrm{FF}_{R,A}$ as mentioned above. This will allow us to consider the corresponding comparison in our generalized context as in the following.
\end{setting}

\begin{setting}
In our current situation, we have our adic version of the corresponding Fargues-Fontaine curves in some rigid analytic deformation context, which includes the situation where the rigid analytic spaces is some Stein space. In our situation the corresponding space will be:
\begin{displaymath}
O_{R,A_\infty(G)}:=\varinjlim_{n\rightarrow \infty}\bigcup_{0<s<r}\mathrm{Spa}(\widetilde{\Pi}^{[s,r]}_{R,A_n(G)},\widetilde{\Pi}^{[s,r],+}_{R,A_n(G)}).	
\end{displaymath}
Then we have the obvious notion of vector bundles over this space, which is Stein in some generalized sense, this is because the corresponding building rings are not noetherian. Then through taking the corresponding quotient we have the corresponding quotient space $\mathrm{FF}_{R,A_\infty(G)}$. We can also define the corresponding vector bundles over the $\infty$-level space while moreover we have the corresponding notions of families of vector bundles $(M_n)_n$ where each $M_n$ is a quasicoherent finite locally free sheaf over $\mathrm{FF}_{R,A_n(G)}$.	
\end{setting}

\indent Then we have the following relative version of \cite[Theorem 8.7.7]{KL15}:

\begin{proposition}
The categories of $A$-relative vector bundles over the two spaces $\mathrm{FF}_{R,A}$ and $\mathrm{Proj}P_{R,A}$ defined above are equivalent. The categories of families of $A$-relative vector bundles over the two spaces $\mathrm{FF}_{R,A_\infty(G)}$ and $\mathrm{Proj}P_{R,A_\infty(G)}$ defined above are equivalent. 
\end{proposition}

\begin{proof}
One could derive this by the same way as in \cite[Theorem 8.7.7]{KL15}.	
\end{proof}

\begin{remark}
One interesting question here is that we can study the corresponding global sections taking the form of $\varprojlim_{n\rightarrow \infty}M_n$ of the families of vector bundles and the corresponding families of Frobenius modules through the equivalences we established so far. 
\end{remark}

\indent First we have the following corollary of the previous proposition:

\begin{corollary}
Categories of the families of Frobenius modules over $\widetilde{\Pi}_{R,A_\infty(G)}$, the families of Frobenius bundles over $\widetilde{\Pi}_{R,A_\infty(G)}$ and the families of vector bundles over the adic Fargues-Fontaine curve $\mathrm{FF}_{R,A_\infty(G)}$ are equivalent.	
\end{corollary}

\begin{remark}
\indent The main issue here is that the corresponding families of Frobenius modules and vector bundles might not have finite generated global sections since in the whole process of taking inverse limit the numbers of the generators might be blowing up.
At this moment we do not know whether the global sections of a family of vector bundles over $\mathrm{FF}_{R,A_\infty(G)}$ are finitely generated. In the situation where one can get the chance to control the largest number of generators during the whole process of taking the inverse limit, one might be able to have the chance to control the finiteness of the global sections. However, this is not trivially correct. The Fargues-Fontaine curves involved are defined over some Fr\'echet-Stein algebras which behave well due to the well-established theory of Fr\'echet-Stein algebras in the noetherian setting. However the Robba rings in the relative setting are highly not noetherian. The corresponding cohomologies of these big modules or families of big modules will eventually then be the corresponding modules or families of modules over Fr\'echet-Stein algebras, which will be definitely easier to control and study since the theory of Fr\'echet-Stein algebras should be able to help us to construct the corresponding Iwasawa theoretic or $K$-theoretic objects. 

\end{remark}

\indent Furthermore we consider the following Hodge-Iwasawa construction under the idea from Kedlaya-Pottharst through the perfectoid subdomain covering mentioned above over any preadic space $X$ defined over $\mathbb{Q}_p$.

\begin{definition}
We define the following $A$-relative sheaves of period rings after \cite[Definition 9.3.3]{KL15}:
\begin{align}
\underline{\underline{\Omega}}_{X,A},\underline{\underline{\Omega}}_{X,A}^\mathrm{int},	\underline{\underline{\widetilde{\Pi}}}_{X,A}^{\mathrm{int},r},\underline{\underline{\widetilde{\Pi}}}_{X,A}^\mathrm{int},\underline{\underline{\widetilde{\Pi}}}_{X,A}^\mathrm{bd,r},\underline{\underline{\widetilde{\Pi}}}_{X,A}^\mathrm{bd},
\underline{\underline{\widetilde{\Pi}}}_{X,A}^r,\underline{\underline{\widetilde{\Pi}}}_{X,A}^I,
\underline{\underline{\widetilde{\Pi}}}_{X,A}^\infty,
\underline{\underline{\widetilde{\Pi}}}_{X,A},
\end{align}
by locally taking the suitable completed product and then glueing. Similarly we have the following Iwasawa sheaves as well:
\begin{align}
\underline{\underline{\Omega}}_{X,A\widehat{\otimes}A_\infty(G)},\underline{\underline{\Omega}}_{X,A\widehat{\otimes}A_\infty(G)}^\mathrm{int},	\underline{\underline{\widetilde{\Pi}}}_{X,A\widehat{\otimes}A_\infty(G)}^{\mathrm{int},r},\underline{\underline{\widetilde{\Pi}}}_{X,A\widehat{\otimes}A_\infty(G)}^\mathrm{int},\underline{\underline{\widetilde{\Pi}}}_{X,A\widehat{\otimes}A_\infty(G)}^{\mathrm{bd},r},\underline{\underline{\widetilde{\Pi}}}_{X,A\widehat{\otimes}A_\infty(G)}^\mathrm{bd},
\end{align}
\begin{align}
\underline{\underline{\widetilde{\Pi}}}_{X,A\widehat{\otimes}A_\infty(G)}^r,\underline{\underline{\widetilde{\Pi}}}_{X,A\widehat{\otimes}A_\infty(G)}^I,
\underline{\underline{\widetilde{\Pi}}}_{X,A\widehat{\otimes}A_\infty(G)}^\infty,
\underline{\underline{\widetilde{\Pi}}}_{X,A\widehat{\otimes}A_\infty(G)}.
\end{align}
\end{definition}

Then we have the following $A$-relative version of \cite[Theorem 9.3.12]{KL15}:

\begin{proposition}
The categories of $A$-relative vector bundles over $\mathrm{FF}_{X,A}$ and Frobenius-equivariant $A$-relative Hodge-Iwasawa finite locally free sheaves over the period sheaves $\underline{\underline{\widetilde{\Pi}}}_{X,A}^\infty,
\underline{\underline{\widetilde{\Pi}}}_{X,A}$ are all equivalent.
\end{proposition}

\begin{proof}
This is by considering the perfectoid covering and the local equivalence we established before.	
\end{proof}

\subsection{Higher Homotopical Geometrized Tamagawa-Iwasawa Theory of Hodge-Iwasawa Modules}
\indent The structure of the corresponding Hodge-Iwasawa modules in our mind has the potential to give us the chance to use them to study the corresponding rational Iwasawa theory. The corresponding relationship we established above could show us the corresponding Iwasawa theory of the relative Frobenius modules over the pro-\'etale site is equivalent to that of the corresponding $B$-pairs (we will also discuss this more in our further study). For instance the latter gives us directly some kind of family version of exponential maps and dual-exponential maps in the higher dimensional situation. We now give some further discussion around our main goals in mind. We first start from the following context:

\begin{setting}
Over a separated rigid analytic space over $\mathbb{Q}_p$ (we assume it to be of finite type as an adic space) $X$ we have the category of the pseudocoherent $\varphi$-sheaves over the sheaf of ring $\underline{\underline{\widetilde{\Pi}}}_{X}$ over the pro-\'etale site $X_\text{pro-\'etale}$. We then use the notation $ChQCoh\Pi\Phi_{X_\text{pro-\'etale}}$ to denote the category of all the chain complexes of objects in the category $QCoh\Pi\Phi_{X_\text{pro-\'etale}}$ which we define it to be the category of all the quasicoherent $\varphi$-sheaves (as in the pseudo-coherent situations) over the ring $\underline{\underline{\widetilde{\Pi}}}_{X}$. 
\end{setting}

\begin{remark}
One should actually regard the category $QCoh\Pi\Phi_{X_\text{pro-\'etale}}$ as the category of all the direct limits of the objects in the category of all the pseudocoherent $\varphi$-sheaves over the sheaf of ring $\underline{\underline{\widetilde{\Pi}}}_{X}$ over the pro-\'etale site $X_\text{pro-\'etale}$. During taking the direct limit we assume that we are actually always carrying the corresponding Frobenius action.  
\end{remark}

\begin{definition}
A complex of objects of $ChQCoh\Pi\Phi_{X_\text{pro-\'etale}}$  taking the form of
\[
\xymatrix@R+0pc@C+0pc{
\ar[r]\ar[r]\ar[r] &... \ar[r]\ar[r]\ar[r] &M^{n-1} \ar[r]\ar[r]\ar[r] &M^n
\ar[r]\ar[r]\ar[r] &M^{n+1}
\ar[r]\ar[r]\ar[r] &...
}
\]	
is called pseudocoherent if it is quasi-isomorphic to a bounded above complex of finite projective objects in $\Pi\Phi_{X_\text{pro-\'etale}}$ (the category of all the pseudo-coherent $\varphi$-sheaves over $\underline{\underline{\widetilde{\Pi}}}_{X}$)
\[
\xymatrix@R+0pc@C+0pc{
\ar[r]\ar[r]\ar[r] &... \ar[r]\ar[r]\ar[r] &F^{m-2} \ar[r]\ar[r]\ar[r] &F^{m-1}
\ar[r]\ar[r]\ar[r] &F^m
\ar[r]\ar[r]\ar[r] &0.
}
\] 
A complex of objects of $ChQCoh\Pi\Phi_{X_\text{pro-\'etale}}$  taking the form of
\[
\xymatrix@R+0pc@C+0pc{
\ar[r]\ar[r]\ar[r] &... \ar[r]\ar[r]\ar[r] &M^{n-1} \ar[r]\ar[r]\ar[r] &M^n
\ar[r]\ar[r]\ar[r] &M^{n+1}
\ar[r]\ar[r]\ar[r] &...
}
\]	
is called perfect if it is quasi-isomorphic to a bounded complex of finite projective objects in $\Pi\Phi_{X_\text{pro-\'etale}}$ (the category of all the pseudo-coherent $\varphi$-sheaves over $\underline{\underline{\widetilde{\Pi}}}_{X}$ as those in \cite[Chapter 8]{KL16})
\[
\xymatrix@R+0pc@C+0pc{
\ar[r]\ar[r]\ar[r] &... \ar[r]\ar[r]\ar[r] &F^{m-2} \ar[r]\ar[r]\ar[r] &F^{m-1}
\ar[r]\ar[r]\ar[r] &F^m
\ar[r]\ar[r]\ar[r] &0.
}
\] 
We then use the notation $D_{\mathrm{pseudo}}\Pi\Phi_{X_\text{pro-\'etale}}$ to denote the category (not the derived one) of all the pseudocoherent complexes of objects in $QCoh\Pi\Phi_{X_\text{pro-\'etale}}$ in the above sense. And we use the notation $D_{\mathrm{perf}}\Pi\Phi_{X_\text{pro-\'etale}}$ for the perfect complexes. We use the notation $D^{\mathrm{cb}}_{\mathrm{pseudo}}\Pi\Phi_{X_\text{pro-\'etale}}$ to denote the category of cohomology bounded pseudocoherent complexes. And we use the notations $D^{\text{dg-flat}}_{\mathrm{perf}}\Pi\Phi_{X_\text{pro-\'etale}}$ to denote the the subcategories consisting of dg-flat complexes. And we use the notations $D^{\text{str}}_{\mathrm{perf}}\Pi\Phi_{X_\text{pro-\'etale}}$ to denote the subcategories consisting of strictly perfect complexes.
\end{definition}

\begin{proposition}
There is a Waldhausen structure over each of
\begin{align}
&D_{\mathrm{pseudo}}\Pi\Phi_{X_\text{pro-\'etale}},D^{\mathrm{cb}}_{\mathrm{pseudo}}\Pi\Phi_{X_\text{pro-\'etale}},\\
&D_{\mathrm{perf}}\Pi\Phi_{X_\text{pro-\'etale}},D^{\mathrm{dg-flat}}_{\mathrm{perf}}\Pi\Phi_{X_\text{pro-\'etale}},	D^{\mathrm{str}}_{\mathrm{perf}}\Pi\Phi_{X_\text{pro-\'etale}}.
\end{align} 	
\end{proposition}

\begin{proof}
We only show this for the first category. We are talking about the Waldhausen category in the sense considered by \cite[Proposition 3.1.1]{Wit3} and \cite{TT1} (namely the corresponding complicial biWaldhausen categories). Recalling from the work \cite[Proposition 3.1.1]{Wit3}, the corresponding criterion is to check that the corresponding smaller full additive subcategory of complexes extracted from an abelian category has the stability under the corresponding shifts and the extension through the exact sequences. But we are talking about pseudocoherent complexes, two out of three property follows from \cite[Tag, 064R]{SP} where one can show this directly by applying \cite[Derived Category, Lemma 16.11]{SP}. Of course note that we are not talking about derived category. In this situation we define the corresponding cofibration to be the corresponding degreewise monomorphism with kernel in the original category in which we are considering, while we define the corresponding weak-equivalences to be quasi-isomorphisms. Then one can check the corresponding conditions in the criterion of Waldhausen category in the sense of \cite[Proposition 3.1.1]{Wit3} and \cite{TT1} satisfy.
\end{proof}



\indent We now consider the corresponding objects over the schematic Fargues-Fontaine curves. By our previous comparison theorems (and the corresponding theorems in \cite{KL16}) on the deformations of the objects we can only consider the scheme theory to extract the corresponding information on the Tamagawa-Iwasawa deformations for the relative $p$-adic Hodge theory over the Robba rings which to some extent very complicated. We can now make the following discussion:

\begin{setting}
Consider any uniform perfect adic Banach algebra $R$ as in our previous consideration. We have the schematic version of the Fargues-Fontaine curve which is denoted by $\mathrm{Proj}P_{R}$. We now use the notation $Mod\mathcal{O}_{\mathrm{Proj}P_{R}}$ to denote the corresponding category of all the sheaves of $\mathcal{O}_{\mathrm{Proj}P_{R}}$-modules over the scheme $\mathrm{Proj}P_{R}$. Then we will use the notation $ChMod\mathcal{O}_{\mathrm{Proj}P_{R}}$ to denote the category of all the complexes of sheaves of $\mathcal{O}_{\mathrm{Proj}P_{R}}$-modules over the scheme $\mathrm{Proj}P_{R}$, with the corresponding derived category $DMod\mathcal{O}_{\mathrm{Proj}P_{R}}$.
\end{setting}

\begin{definition}
A complex of objects of $Mod\mathcal{O}_{\mathrm{Proj}P_{R}}$ taking the form of
\[
\xymatrix@R+0pc@C+0pc{
\ar[r]\ar[r]\ar[r] &... \ar[r]\ar[r]\ar[r] &M^{n-1} \ar[r]\ar[r]\ar[r] &M^n
\ar[r]\ar[r]\ar[r] &M^{n+1}
\ar[r]\ar[r]\ar[r] &...
}
\]	
is called pseudocoherent if it is quasi-isomorphic to a bounded above complex of finite projective objects in the category of all the sheaves of $\mathcal{O}_{\mathrm{Proj}P_{R}}$-modules over the schematic Fargues-Fontaine curve
\[
\xymatrix@R+0pc@C+0pc{
\ar[r]\ar[r]\ar[r] &... \ar[r]\ar[r]\ar[r] &F^{m-2} \ar[r]\ar[r]\ar[r] &F^{m-1}
\ar[r]\ar[r]\ar[r] &F^m
\ar[r]\ar[r]\ar[r] &0.
}
\]
We then use the notation $D_{\mathrm{pseudo}}\mathrm{Proj}P_{R}$ to denote the category (not the derived one) of all the pseudocoherent complexes of objects in $Mod\mathcal{O}_{\mathrm{Proj}P_{R}}$ in the above sense. A complex of objects of $Mod\mathcal{O}_{\mathrm{Proj}P_{R}}$ taking the form of
\[
\xymatrix@R+0pc@C+0pc{
\ar[r]\ar[r]\ar[r] &... \ar[r]\ar[r]\ar[r] &M^{n-1} \ar[r]\ar[r]\ar[r] &M^n
\ar[r]\ar[r]\ar[r] &M^{n+1}
\ar[r]\ar[r]\ar[r] &...
}
\]	
is called perfect if it is quasi-isomorphic to a bounded complex of finite projective objects in the category of all the sheaves of $\mathcal{O}_{\mathrm{Proj}P_{R}}$-modules over the schematic Fargues-Fontaine curve
\[
\xymatrix@R+0pc@C+0pc{
\ar[r]\ar[r]\ar[r] &... \ar[r]\ar[r]\ar[r] &F^{m-2} \ar[r]\ar[r]\ar[r] &F^{m-1}
\ar[r]\ar[r]\ar[r] &F^m
\ar[r]\ar[r]\ar[r] &0.
}
\]
We then use the notation $D_{\mathrm{perf}}\mathrm{Proj}P_{R}$ to denote the category (not the derived one) of all the perfect complexes of objects in $Mod\mathcal{O}_{\mathrm{Proj}P_{R}}$ in the above sense.	
\end{definition}

\begin{proposition}
The category $D_{\mathrm{pseudo}}\mathrm{Proj}P_{R}$ is a category admitting Waldhausen structure.	
\end{proposition}

\begin{proof}
This is standard such as in \cite[Section 1-3]{TT1}. We present the proof as above for the convenience of the readers. We are now talking about again the context of \cite[Proposition 3.1.1]{Wit3} and \cite{TT1}. The corresponding criterion in \cite{Wit3} requires the corresponding stability of the corresponding shifts and extension in the category of the corresponding complexes in our context. Again we can consider the corresponding results of \cite[Tag, 064R]{SP} to prove this. Then the corresponding cofibration will be the corresponding monomorphism with kernel in the original category of the complexes while the corresponding weak-equivalences are taken to be the corresponding quasi-isomorphisms. 	
\end{proof}

Then we can define the corresponding category of pseudocoherent sheaves over the adic Fargues-Fontaine curves:

\begin{setting}
After \cite{KL16}, we consider the cateogory $Mod\mathcal{O}_{\mathrm{FF}_{\widetilde{R}_\psi}}$ of all sheaves of $\mathcal{O}_{\mathrm{FF}_{\widetilde{R}_\psi}}$-modules. Here the ring $\widetilde{R}_\psi$ is the perfect adic uniform Banach ring attached to the toric tower considered in \cite{KL16}. We then use the notation $ChMod\mathcal{O}_{\mathrm{FF}_{\widetilde{R}_\psi}}$ to denote the category of all the chain complexes of objects in the previous category $Mod\mathcal{O}_{\mathrm{FF}_{\widetilde{R}_\psi}}$. 
\end{setting}

\begin{definition}
A complex of objects of $ChMod\mathcal{O}_{\mathrm{FF}_{\widetilde{R}_\psi}}$ taking the form of
\[
\xymatrix@R+0pc@C+0pc{
\ar[r]\ar[r]\ar[r] &... \ar[r]\ar[r]\ar[r] &M^{n-1} \ar[r]\ar[r]\ar[r] &M^n
\ar[r]\ar[r]\ar[r] &M^{n+1}
\ar[r]\ar[r]\ar[r] &...
}
\]	
is called pseudocoherent if it is quasi-isomorphic to a bounded above complex of finite projective objects in the category of all the pseudocoherent sheaves over the Fargues-Fontaine curve
\[
\xymatrix@R+0pc@C+0pc{
\ar[r]\ar[r]\ar[r] &... \ar[r]\ar[r]\ar[r] &F^{m-2} \ar[r]\ar[r]\ar[r] &F^{m-1}
\ar[r]\ar[r]\ar[r] &F^m
\ar[r]\ar[r]\ar[r] &0.
}
\] 
We then use the notation $D_{\mathrm{pseudo,alg}}\mathrm{FF}_{\widetilde{R}_\psi}$ to denote the category (not the derived one) of all the pseudocoherent complexes of objects in $Mod\mathcal{O}_{\mathrm{FF}_{\widetilde{R}_\psi}}$ in the above sense. We also have the category $D^\mathrm{cb}_{\mathrm{pseudo,alg}}\mathrm{FF}_{\widetilde{R}_\psi}$ consisting of all the cohomology bounded pseudocoherent complexes.\\
A complex of objects of $ChMod\mathcal{O}_{\mathrm{FF}_{\widetilde{R}_\psi}}$ taking the form of
\[
\xymatrix@R+0pc@C+0pc{
\ar[r]\ar[r]\ar[r] &... \ar[r]\ar[r]\ar[r] &M^{n-1} \ar[r]\ar[r]\ar[r] &M^n
\ar[r]\ar[r]\ar[r] &M^{n+1}
\ar[r]\ar[r]\ar[r] &...
}
\]	
is called perfect if it is quasi-isomorphic to a bounded complex of finite projective objects in the category of all the pseudocoherent sheaves over the Fargues-Fontaine curve
\[
\xymatrix@R+0pc@C+0pc{
\ar[r]\ar[r]\ar[r] &... \ar[r]\ar[r]\ar[r] &F^{m-2} \ar[r]\ar[r]\ar[r] &F^{m-1}
\ar[r]\ar[r]\ar[r] &F^m
\ar[r]\ar[r]\ar[r] &0.
}
\] 
We then use the notation $D_{\mathrm{perf,alg}}\mathrm{FF}_{\widetilde{R}_\psi}$ to denote the category (not the derived one) of all the perfect complexes of objects in $Mod\mathcal{O}_{\mathrm{FF}_{\widetilde{R}_\psi}}$ in the above sense. We also have the subcategories $D^{\mathrm{dg-flat}}_{\mathrm{perf,alg}}\mathrm{FF}_{\widetilde{R}_\psi}$ and $D^{\mathrm{str}}_{\mathrm{perf,alg}}\mathrm{FF}_{\widetilde{R}_\psi}$ of dg-flat objects and strictly perfect objects.
\end{definition}

\begin{proposition}
The categories
\begin{align}
&D_{\mathrm{pseudo,alg}}\mathrm{FF}_{\widetilde{R}_\psi},D^\mathrm{cb}_{\mathrm{pseudo,alg}}\mathrm{FF}_{\widetilde{R}_\psi}\\
&D_{\mathrm{perf,alg}}\mathrm{FF}_{\widetilde{R}_\psi},D^\mathrm{dg-flat}_{\mathrm{perf,alg}}\mathrm{FF}_{\widetilde{R}_\psi},D^\mathrm{str}_{\mathrm{perf,alg}}\mathrm{FF}_{\widetilde{R}_\psi}	
\end{align}
admit structure of Waldhausen categories.	
\end{proposition}


\begin{proof}
The corresponding category involved in this situation is not that far from being equivalent to the situation where we considered in the Frobenius sheaves situation. On the other hand, one can then following the proof to give the chance to prove the corresponding statement here.	
\end{proof}

\indent Now we consider the more general analytic objects:

\begin{setting}
Following \cite{KL16}, we have the category $QCoh\mathcal{O}_{\mathrm{FF}_{\widetilde{R}_\psi}}$ of the quasi-coherent sheaves (direct limits of the pseudocoherent objects) over the Fargues-Fontaine curve $\mathrm{FF}_{\widetilde{R}_\psi}$. We then use the notation $ChQCoh\mathcal{O}_{\mathrm{FF}_{\widetilde{R}_\psi}}$ to denote the category of all the chain complexes of objects in the previous category $QCoh\mathcal{O}_{\mathrm{FF}_{\widetilde{R}_\psi}}$. 
\end{setting}

\begin{definition}
A complex of objects of $ChQCoh\mathcal{O}_{\mathrm{FF}_{\widetilde{R}_\psi}}$ taking the form of
\[
\xymatrix@R+0pc@C+0pc{
\ar[r]\ar[r]\ar[r] &... \ar[r]\ar[r]\ar[r] &M^{n-1} \ar[r]\ar[r]\ar[r] &M^n
\ar[r]\ar[r]\ar[r] &M^{n+1}
\ar[r]\ar[r]\ar[r] &...
}
\]	
is called pseudocoherent if it is quasi-isomorphic to a bounded above complex of finite projective objects in the category of all the pseudocoherent sheaves over the Fargues-Fontaine curve:
\[
\xymatrix@R+0pc@C+0pc{
\ar[r]\ar[r]\ar[r] &... \ar[r]\ar[r]\ar[r] &F^{m-2} \ar[r]\ar[r]\ar[r] &F^{m-1}
\ar[r]\ar[r]\ar[r] &F^m
\ar[r]\ar[r]\ar[r] &0.
}
\]
We then use the notation $D_{\mathrm{pseudo}}\mathrm{FF}_{\widetilde{R}_\psi}$ to denote the category (not the derived one) of all the pseudocoherent complexes of objects in $QCoh\mathcal{O}_{\mathrm{FF}_{\widetilde{R}_\psi}}$ in the above sense. And we use the notations $D^\mathrm{cb}_{\mathrm{pseudo}}\mathrm{FF}_{\widetilde{R}_\psi}$ to denote the subcategories of cohomology bounded complexes.
A complex of objects of $ChQCoh\mathcal{O}_{\mathrm{FF}_{\widetilde{R}_\psi}}$ taking the form of
\[
\xymatrix@R+0pc@C+0pc{
\ar[r]\ar[r]\ar[r] &... \ar[r]\ar[r]\ar[r] &M^{n-1} \ar[r]\ar[r]\ar[r] &M^n
\ar[r]\ar[r]\ar[r] &M^{n+1}
\ar[r]\ar[r]\ar[r] &...
}
\]	
is called perfect if it is quasi-isomorphic to a bounded complex of finite projective objects in the category of all the pseudocoherent sheaves over the Fargues-Fontaine curve:
\[
\xymatrix@R+0pc@C+0pc{
\ar[r]\ar[r]\ar[r] &... \ar[r]\ar[r]\ar[r] &F^{m-2} \ar[r]\ar[r]\ar[r] &F^{m-1}
\ar[r]\ar[r]\ar[r] &F^m
\ar[r]\ar[r]\ar[r] &0.
}
\]
We then use the notation $D_{\mathrm{perf}}\mathrm{FF}_{\widetilde{R}_\psi}$ to denote the category (not the derived one) of all the perfect complexes of objects in $QCoh\mathcal{O}_{\mathrm{FF}_{\widetilde{R}_\psi}}$ in the above sense. And we use the notations $D^\mathrm{dg-flat}_{\mathrm{perf}}\mathrm{FF}_{\widetilde{R}_\psi}$ and $D^\mathrm{str}_{\mathrm{perf}}\mathrm{FF}_{\widetilde{R}_\psi}$ to denote the subcategories of dg-flat complexes and strictly perfect complexes.
\end{definition}

\begin{proposition}
The categories
\begin{align}
&D_{\mathrm{pseudo}}\mathrm{FF}_{\widetilde{R}_\psi},D^{\mathrm{cb}}_{\mathrm{pseudo}}\mathrm{FF}_{\widetilde{R}_\psi},\\
&D_{\mathrm{perf}}\mathrm{FF}_{\widetilde{R}_\psi},D^{\mathrm{dg-flat}}_{\mathrm{perf}}\mathrm{FF}_{\widetilde{R}_\psi},D^{\mathrm{str}}_{\mathrm{perf}}\mathrm{FF}_{\widetilde{R}_\psi}
\end{align}
admit the structure of Waldhausen categories.	
\end{proposition}

\begin{proof}
See the proof for the previous proposition.	
\end{proof}

\indent Now one can further discuss the corresponding deformed version of the picture above, we mainly then focus on the schematic version of the Fargues-Fontaine curve:

\begin{setting}
Let $A$ be a reduced affinoid algebra as before. Consider any uniform perfect adic Banach algebra $R$ as in our previous consideration. We have the schematic version of the Fargues-Fontaine curve which is denoted by $\mathrm{Proj}P_{R,A}$. We now use the notation $Mod\mathcal{O}_{\mathrm{Proj}P_{R,A}}$ to denote the corresponding category of all the sheaves of $\mathcal{O}_{\mathrm{Proj}P_{R,A}}$-modules over the scheme $\mathrm{Proj}P_{R}$. Then we will use the notation $ChMod\mathcal{O}_{\mathrm{Proj}P_{R,A}}$ to denote the category of all the complexes of sheaves of $\mathcal{O}_{\mathrm{Proj}P_{R,A}}$-modules over the scheme $\mathrm{Proj}P_{R,A}$, with the corresponding derived category $DMod\mathcal{O}_{\mathrm{Proj}P_{R,A}}$.
\end{setting}

\begin{definition}
A complex of objects of $Mod\mathcal{O}_{\mathrm{Proj}P_{R,A}}$ taking the form of
\[
\xymatrix@R+0pc@C+0pc{
\ar[r]\ar[r]\ar[r] &... \ar[r]\ar[r]\ar[r] &M^{n-1} \ar[r]\ar[r]\ar[r] &M^n
\ar[r]\ar[r]\ar[r] &M^{n+1}
\ar[r]\ar[r]\ar[r] &...
}
\]	
is called pseudocoherent if it is quasi-isomorphic to a bounded above complex of finite projective objects in the category of all the quasicoherent sheaves over the schematic Fargues-Fontaine curve
\[
\xymatrix@R+0pc@C+0pc{
\ar[r]\ar[r]\ar[r] &... \ar[r]\ar[r]\ar[r] &F^{m-2} \ar[r]\ar[r]\ar[r] &F^{m-1}
\ar[r]\ar[r]\ar[r] &F^m
\ar[r]\ar[r]\ar[r] &0.
}
\]
We then use the notation $D_{\mathrm{pseudo}}\mathrm{Proj}P_{R,A}$ to denote the category (not the derived one) of all the pseudocoherent complexes of objects in $Mod\mathcal{O}_{\mathrm{Proj}P_{R,A}}$ in the above sense. We can also define the corresponding category $D^{\mathrm{cb}}_{\mathrm{pseudo}}\mathrm{Proj}P_{R,A}$ as above. \\	
\indent A complex of objects of $Mod\mathcal{O}_{\mathrm{Proj}P_{R,A}}$ taking the form of
\[
\xymatrix@R+0pc@C+0pc{
\ar[r]\ar[r]\ar[r] &... \ar[r]\ar[r]\ar[r] &M^{n-1} \ar[r]\ar[r]\ar[r] &M^n
\ar[r]\ar[r]\ar[r] &M^{n+1}
\ar[r]\ar[r]\ar[r] &...
}
\]	
is called perfect if it is quasi-isomorphic to a bounded complex of finite projective objects in the category of all the quasicoherent sheaves over the schematic Fargues-Fontaine curve
\[
\xymatrix@R+0pc@C+0pc{
\ar[r]\ar[r]\ar[r] &... \ar[r]\ar[r]\ar[r] &F^{m-2} \ar[r]\ar[r]\ar[r] &F^{m-1}
\ar[r]\ar[r]\ar[r] &F^m
\ar[r]\ar[r]\ar[r] &0.
}
\]
We then use the notation $D_{\mathrm{perf}}\mathrm{Proj}P_{R,A}$ to denote the category (not the derived one) of all the perfect complexes of objects in $Mod\mathcal{O}_{\mathrm{Proj}P_{R,A}}$ in the above sense. We can also define the corresponding category $D^{\mathrm{dg-flat}}_{\mathrm{perf}}\mathrm{Proj}P_{R,A}$ of dg-flat perfect complexes and the corresponding category $D^{\mathrm{str}}_{\mathrm{pseudo}}\mathrm{Proj}P_{R,A}$ of strictly perfect complexes.
\end{definition}

\begin{proposition}
The categories
\begin{align}
&D_{\mathrm{pseudo}}\mathrm{Proj}P_{R,A},D^{\mathrm{cb}}_{\mathrm{pseudo}}\mathrm{Proj}P_{R,A},\\
&D_{\mathrm{perf}}\mathrm{Proj}P_{R,A},D^{\mathrm{dg-flat}}_{\mathrm{perf}}\mathrm{Proj}P_{R,A},D^{\mathrm{str}}_{\mathrm{perf}}\mathrm{Proj}P_{R,A},	
\end{align}
are categories admitting Waldhausen structure.	
\end{proposition}

\begin{proof}
We only prove the statement for the first category. See the proof above for $D_{\mathrm{pseudo}}\mathrm{Proj}P_{R}$ without the deformation.	
\end{proof}

\begin{conjecture}
Assume that $\psi$ is the cyclotomic tower. Then the corresponding total derived section functor (by using the Godement resolution) induces a map on the $K$-theory space $\mathbb{K}D^\mathrm{dg-flat}_{\mathrm{perf}}\mathrm{Proj}P_{\widetilde{R}_\psi,A}$ to that of the category $\mathbb{K}D^{\mathrm{dg-flat}}_{\mathrm{perf}}(A)$ of all the dg-flat perfect complexes of $A$-modules. And we conjecture that this is homotopic to zero. \\
\indent Assume that $\psi$ is the cyclotomic tower. Then the corresponding total derived section functor (by using the Godement resolution) induces a map on the $K$-theory space $\mathbb{K}D^\mathrm{str}_{\mathrm{perf}}\mathrm{Proj}P_{\widetilde{R}_\psi,A}$ to that of the category $\mathbb{K}D^\mathrm{str}_{\mathrm{perf}}(A)$ of all the strictly perfect complexes of $A$-modules. 
And we conjecture that in this situation this is homotopic to zero.
\end{conjecture}

\indent Although not equivalent to the previous picture in the deformed setting we should still consider the following picture over the adic Fargues-Fontaine curve:

\begin{setting}
Let $A$ be as above such that we have the well-defined adic space $\mathrm{FF}_{\widetilde{R}_\psi,A}$	where the corresponding tower $\psi$ is the same as the previous one. Then we have the corresponding category $\mathrm{Mod}\mathcal{O}_{\mathrm{FF}_{\widetilde{R}_\psi,A}}$ as above by taking the corresponding sheaves of $\mathcal{O}_{\mathrm{FF}_{\widetilde{R}_\psi,A}}$-modules over the deformed version of the Fargues-Fontaine curve. Then we have the corresponding category of all the chain complexes namely the category $Ch\mathrm{Mod}\mathcal{O}_{\mathrm{FF}_{\widetilde{R}_\psi,A}}$ and the corresponding derived one $D\mathrm{Mod}\mathcal{O}_{\mathrm{FF}_{\widetilde{R}_\psi,A}}$.
\end{setting}

\begin{definition}
A complex in the category $Ch\mathrm{Mod}\mathcal{O}_{\mathrm{FF}_{\widetilde{R}_\psi,A}}$ taking the form of 
\[
\xymatrix@R+0pc@C+0pc{
\ar[r]\ar[r]\ar[r] &... \ar[r]\ar[r]\ar[r] &M^{n-1} \ar[r]\ar[r]\ar[r] &M^n
\ar[r]\ar[r]\ar[r] &M^{n+1}
\ar[r]\ar[r]\ar[r] &...
}
\]		
is now called pseudocoherent if it is quasi-isomorphic to a bounded above complex of finite projective objects in the category of all the pseudocoherent objects defined over the deformed version of adic version of the Fargues-Fontaine curve in the current situation.
Correspondingly we use the notation $D_\mathrm{pseudo,alg}\mathrm{FF}_{\widetilde{R}_\psi,A}$ to denote the corresponding category of all the pseudocoherent complexes defined above. We also have the corresponding category as above $D^\mathrm{cb}_\mathrm{pseudo,alg}\mathrm{FF}_{\widetilde{R}_\psi,A}$. 
A complex in the category $Ch\mathrm{Mod}\mathcal{O}_{\mathrm{FF}_{\widetilde{R}_\psi,A}}$ taking the form of 
\[
\xymatrix@R+0pc@C+0pc{
\ar[r]\ar[r]\ar[r] &... \ar[r]\ar[r]\ar[r] &M^{n-1} \ar[r]\ar[r]\ar[r] &M^n
\ar[r]\ar[r]\ar[r] &M^{n+1}
\ar[r]\ar[r]\ar[r] &...
}
\]		
is now called perfect if it is quasi-isomorphic to a bounded complex of finite projective objects in the category of all the pseudocoherent objects defined over the deformed version of adic version of the Fargues-Fontaine curve in the current situation.
Correspondingly we use the notation $D_\mathrm{perf,alg}\mathrm{FF}_{\widetilde{R}_\psi,A}$ to denote the corresponding category of all the pseudocoherent complexes defined above. We also have the corresponding categories of dg-flat complexes and the strictly perfect complexes which will be denoted by $D^\mathrm{dg-flat}_\mathrm{perf,alg}\mathrm{FF}_{\widetilde{R}_\psi,A}$ and $D^\mathrm{str}_\mathrm{perf,alg}\mathrm{FF}_{\widetilde{R}_\psi,A}$. 
\end{definition}

\begin{proposition}
The categories 
\begin{align}
&D_\mathrm{pseudo,alg}\mathrm{FF}_{\widetilde{R}_\psi,A},D^\mathrm{cb}_\mathrm{pseudo,alg}\mathrm{FF}_{\widetilde{R}_\psi,A},\\
&D_\mathrm{perf,alg}\mathrm{FF}_{\widetilde{R}_\psi,A},D^\mathrm{dg-flat}_\mathrm{perf,alg}\mathrm{FF}_{\widetilde{R}_\psi,A},D^\mathrm{str}_\mathrm{perf,alg}\mathrm{FF}_{\widetilde{R}_\psi,A}	
\end{align}
admit Waldhausen structure.	
\end{proposition}

\begin{proof}
See the proof before where we do not have the corresponding deformation with respect to the algebra $A$.
\end{proof}

\begin{setting}
We use the notation $QCoh\mathrm{FF}_{\widetilde{R}_\psi,A}$ to denote the corresponding quasi-coherent sheaves of $\mathcal{O}_{\mathrm{FF}_{\widetilde{R}_\psi,A}}$-modules (direct limits of the pseudocoherent ones which are assumed to form an abelian category) over the adic space $\mathrm{FF}_{\widetilde{R}_\psi,A}$. Then we use the notation $ChQCoh\mathrm{FF}_{\widetilde{R}_\psi,A}$ to denote the category of chain complexes consisting of all the objects in the category $QCoh\mathrm{FF}_{\widetilde{R}_\psi,A}$. And we use the notation $DQCoh\mathrm{FF}_{\widetilde{R}_\psi,A}$ to denote the corresponding derived category. 
\end{setting}

\begin{definition}
A complex in the category $ChQCoh\mathrm{FF}_{\widetilde{R}_\psi,A}$ taking the form of 
\[
\xymatrix@R+0pc@C+0pc{
\ar[r]\ar[r]\ar[r] &... \ar[r]\ar[r]\ar[r] &M^{n-1} \ar[r]\ar[r]\ar[r] &M^n
\ar[r]\ar[r]\ar[r] &M^{n+1}
\ar[r]\ar[r]\ar[r] &...
}
\]
is called then pseudocoherent if it is quasi-isomorphic to a bounded above complex of finite projective objects in the category of all the pseudocoherent objects over the deformed version of the adic version of the Fargues-Fontaine curve. The corresponding whole category of all such objects is denoted now by the notation $D_{\mathrm{pseudo}}\mathrm{FF}_{\widetilde{R}_\psi,A}$. We also have the category as above $D^{\mathrm{cb}}_{\mathrm{pseudo}}\mathrm{FF}_{\widetilde{R}_\psi,A}$.	
\end{definition}

\begin{definition}
A complex in the category $ChQCoh\mathrm{FF}_{\widetilde{R}_\psi,A}$ taking the form of 
\[
\xymatrix@R+0pc@C+0pc{
\ar[r]\ar[r]\ar[r] &... \ar[r]\ar[r]\ar[r] &M^{n-1} \ar[r]\ar[r]\ar[r] &M^n
\ar[r]\ar[r]\ar[r] &M^{n+1}
\ar[r]\ar[r]\ar[r] &...
}
\]
is called then perfect if it is quasi-isomorphic to a bounded complex of finite projective objects in the category of all the pseudocoherent objects over the deformed version of the adic version of the Fargues-Fontaine curve. The corresponding whole category of all such objects is denoted now by the notation $D_{\mathrm{perf}}\mathrm{FF}_{\widetilde{R}_\psi,A}$. We also have the categories of dg-flat complexes and strictly perfect complexes, namely $D^{\mathrm{dg-flat}}_{\mathrm{perf}}\mathrm{FF}_{\widetilde{R}_\psi,A}$ and $D^{\mathrm{str}}_{\mathrm{perf}}\mathrm{FF}_{\widetilde{R}_\psi,A}$.	
\end{definition}

\begin{remark}
The consideration here is in some sense different from the corresponding consideration in the schematic Fargues-Fontaine curve situation where we use the machinery of scheme theory and it is sometimes not stable under the analytic operations. But the latter should have its own interests from algebraic geometric point of view, in the general situation. For instance this gives us the chance to study the corresponding mysterious algebraic modules over the Robba rings with Frobenius pullbacks as isomorphisms. We remind the readers of the fact that over the cyclotomic tower many difficulties disappears. For instance more general algebraic approaches could be available.
\end{remark}

\begin{proposition}
The categories 
\begin{align}
&D_\mathrm{pseudo}\mathrm{FF}_{\widetilde{R}_\psi,A},D^\mathrm{cb}_\mathrm{pseudo}\mathrm{FF}_{\widetilde{R}_\psi,A},\\
&D_\mathrm{perf}\mathrm{FF}_{\widetilde{R}_\psi,A},D^\mathrm{dg-flat}_\mathrm{perf}\mathrm{FF}_{\widetilde{R}_\psi,A},D^\mathrm{str}_\mathrm{perf}\mathrm{FF}_{\widetilde{R}_\psi,A}	
\end{align}
admit Waldhausen structure.		
\end{proposition}

\begin{proof}
See the proof of the previous propositions around the same issue.	
\end{proof}

\begin{conjecture}
Assume that $\psi$ is the cyclotomic tower. Then the corresponding derived section functor induces a map on the $K$-theory space $\mathbb{K}D^\mathrm{dg-flat}_{\mathrm{perf}}\mathrm{FF}_{\widetilde{R}_\psi,A}$ to that of the category $D^\mathrm{dg-flat}_{\mathrm{perf}}(A)$ of all the dg-flat perfect complexes of $A$-modules. We conjecture that this is homotopic to zero.
 
Assume that $\psi$ is the cyclotomic tower. Then the corresponding derived section functor induces a map on the $K$-theory space $\mathbb{K}D^\mathrm{str}_{\mathrm{perf}}\mathrm{FF}_{\widetilde{R}_\psi,A}$ to that of the category $D^\mathrm{str}_{\mathrm{perf}}(A)$ of all the strictly perfect complexes of $A$-modules. We conjecture that this is homotopic to zero. 
\end{conjecture}

\begin{remark}
These are conjectured in the situation where $\psi$ is just the cyclotomic tower, in general one could easily guess what the conjectures look like.	
\end{remark}

\indent We make some discussion here on the picture in the \'etale setting. The conjecture is much easier in the $\ell$-adic situation. This represents the difference between the $\ell$-adic situation and the $p$-adic setting in some situation. For instance let us try to consider the following proposition after Witte (see \cite[Proposition 6.1.5]{Wit3}):

\begin{proposition}
Let $T$ be $\mathbb{Z}_\ell$ where $\ell\neq p$. For $X$ a rigid analytic space over the $p$-adic field $\mathbb{Q}_p$ which is separated and of finite type (when considered as some adic space locally of finite type, see \cite{Huber1}), suppose we have that for $\sharp=\text{\'et}$ the category $\mathbb{D}_{\mathrm{perf}}(X_\sharp,T)$ could be endowed with the structure of Waldhausen categories and there is a Waldhausen exact functor $R\Gamma_?(X_\sharp,.)$ (induced by the direct image functor in this context as in \cite[Chapter 4-5]{Wit3}) as below for $\sharp=\text{\'et}$ and $?=\emptyset$:
\[
\xymatrix@R+0pc@C+3pc{
\mathbb{D}_{\mathrm{perf}}(X_\sharp,T)\ar[r]^{R\Gamma_?(X_\sharp,.)}\ar[r]\ar[r] &\mathbb{D}_{\mathrm{perf}}(T)
}
\]
which induces the corresponding the morphism $\mathbb{K}R\Gamma_?(X_\sharp,.)$ as:
\[
\xymatrix@R+0pc@C+3pc{
\mathbb{K}\mathbb{D}_{\mathrm{perf}}(X_\sharp,T)\ar[r]^{\mathbb{K}R\Gamma_?(X_\sharp,.)}\ar[r]\ar[r] & \mathbb{K}\mathbb{D}_{\mathrm{perf}}(T).
}
\]	
Then there is a homotopy between this morphism and the zero map:
\[
\xymatrix@R+0pc@C+3pc{
\mathbb{K}\mathbb{D}_{\mathrm{perf}}(X_\sharp,T)\ar[r]^{0}\ar[r]\ar[r] & \mathbb{K}\mathbb{D}_{\mathrm{perf}}(T).
}
\]
\end{proposition}

\begin{proof}
We follow the strategy of \cite[Proposition 6.1.5]{Wit3} to prove this. The idea is to pullback all the things back along the following morphism:
\begin{displaymath}
f:\mathrm{Spa}(\mathbb{Q}^\mathrm{ur}_p,\mathfrak{o}_{\mathbb{Q}^\mathrm{ur}_p})\rightarrow \mathrm{Spa}(\mathbb{Q}_p,\mathfrak{o}_{\mathbb{Q}_p}).	
\end{displaymath}
Now for each complex $M^\bullet$ involved we put:
\begin{align}
R\Gamma(X_{\mathbb{Q}^\mathrm{ur}_p,\sharp},M^\bullet):=\Gamma(\mathrm{Spa}(\mathbb{Q}^\mathrm{ur}_p,\mathfrak{o}_{\mathbb{Q}^\mathrm{ur}_p}),f^*R\pi_* M^\bullet),
\end{align}
where the morphism $\pi:X\rightarrow \mathrm{Spa}(\mathbb{Q}_p,\mathfrak{o}_{\mathbb{Q}_p})$ is the corresponding structure morphism of $X$. We have the corresponding finiteness due to our assumption here. One then looks at the corresponding Frobenius $\mathrm{Fr}$ over the corresponding field $\mathbb{Q}^\mathrm{ur}_p$. Then as in \cite[Proposition 6.1.5]{Wit3} we have that then the following cofibration sequence for each complex $M^\bullet$:
\[
\xymatrix@R+0pc@C+3pc{
R\Gamma_?(X_{\mathbb{Q}_p,\sharp},M^\bullet)\ar[r]\ar[r]\ar[r] &R\Gamma_?(X_{\mathbb{Q}^\mathrm{ur}_p,\sharp},M^\bullet)\ar[r]^{1-\mathrm{Fr}}\ar[r] &R\Gamma_?(X_{\mathbb{Q}^\mathrm{ur}_p,\sharp},M^\bullet),
}
\]
as in \cite[Proposition 6.1.5]{Wit3} (note that here we need the corresponding flabbiness of the corresponding derived direct images), by the corresponding Waldhausen's additive theorem we have the corresponding homotopy:
\begin{align}
\mathbb{K}R\Gamma_?(X_{\mathbb{Q}^\mathrm{ur}_p,\sharp},M^\bullet)\oplus \mathbb{K}R\Gamma_?(X_{\mathbb{Q}_p,\sharp},M^\bullet)\leadsto \mathbb{K}R\Gamma_?(X_{\mathbb{Q}^\mathrm{ur}_p,\sharp},M^\bullet)	
\end{align}
which implies that we have then as in \cite[Proposition 6.1.5]{Wit3}
\begin{align}
\pi_i\mathbb{K}R\Gamma_?(X_\sharp,.)=0.	
\end{align}
\end{proof}

\indent We now show how this recovers the results of Witte at a point:

\begin{corollary}
Suppose now the space $X$ is just the point $\mathrm{Spa}(\mathbb{Q}_p,\mathfrak{o}_{\mathbb{Q}_p})$, and keep the assumption that $T$ is the ring $\mathbb{Z}_\ell$. Then we have for $\sharp=\text{\'et}$ the Waldhausen category $\mathbb{D}_{\mathrm{perf}}(X_\sharp,T)$ which induces the Waldhausen exact functor (induced from the total derived direct image functor) $R\Gamma_?(X_\sharp,.)$ where $?=\emptyset$ and $\sharp$ as above:
\[
\xymatrix@R+0pc@C+3pc{
\mathbb{D}_{\mathrm{perf}}(X_\sharp,T)\ar[r]^{R\Gamma_?(X_\sharp,.)}\ar[r]\ar[r] &\mathbb{D}_{\mathrm{perf}}(T)
}
\]
which induces the corresponding the morphism $\mathbb{K}R\Gamma_?(X_\sharp,.)$ as:
\[
\xymatrix@R+0pc@C+3pc{
\mathbb{K}\mathbb{D}_{\mathrm{perf}}(X_\sharp,T)\ar[r]^{\mathbb{K}R\Gamma_?(X_\sharp,.)}\ar[r]\ar[r] & \mathbb{K}\mathbb{D}_{\mathrm{perf}}(T).
}
\]	
And moreover we have that this map is homotopic to zero:
\[
\xymatrix@R+0pc@C+3pc{
\mathbb{K}\mathbb{D}_{\mathrm{perf}}(X_\sharp,T)\ar[r]^{0}\ar[r]\ar[r] & \mathbb{K}\mathbb{D}_{\mathrm{perf}}(T).
}
\]
	
\end{corollary}

\begin{proof}
One can apply results of \cite{Wit3} to recover the corresponding structures here or construct the corresponding structures here directly. Then we can apply the previous proposition. 
\end{proof}

\indent Let us analyze the corresponding $p$-adic setting in the following sense. Since in the proof used above the corresponding results on the finiteness the cohomological dimension over the space attached to the maximal unramified extension of $\mathbb{Q}_p$ are crucial. We first consider the integral version of the picture we considered above.

\begin{setting}
We are going to use the notation $QCoh\Omega^{\mathrm{int}}\Phi^a_{X_{\text{pro-\'etale}}}$ (where $X$ is assumed to be separated and of finite type as an adic space) to denote the corresponding quasi-coherent (direct limits of pseudocoherent ones) $\varphi^a$-sheaves over the sheaf of ring $\underline{\underline{{\Omega}}}^\mathrm{int}_{X_{\text{pro-\'etale}}}$, and we use the corresponding notation $ChQCoh\Omega^{\mathrm{int}}\Phi^a_{X_{\text{pro-\'etale}}}$ to denote the corresponding category of the chain complexes and use the corresponding notation $DQCoh\Omega^{\mathrm{int}}\Phi^a_{X_{\text{pro-\'etale}}}$ to denote the corresponding derived category.	
\end{setting}

\begin{definition}
A complex considered in the previous setting taking the form of 
\[
\xymatrix@R+0pc@C+0pc{
\ar[r]\ar[r]\ar[r] &... \ar[r]\ar[r]\ar[r] &M^{n-1} \ar[r]\ar[r]\ar[r] &M^n
\ar[r]\ar[r]\ar[r] &M^{n+1}
\ar[r]\ar[r]\ar[r] &...
}
\]	
is called then pseudocoherent if it is quasi-isomorphic to a bounded above complex of finite projective objects in the category of pseudocoherent $\varphi^a$-modules over $\underline{\underline{{\Omega}}}^\mathrm{int}_{X}$.	
\end{definition}




The corresponding setting of local systems gives us the following consideration after Witte:

\begin{conjecture}
Assume that $T$ is an adic ring over $\mathbb{Z}_p$ such that we can find a two-sided ideal $I$ such that each $T/I^n$ for $n\geq 0$ is finite of order a power of $p$. Suppose that $X$ is a rigid analytic space which is assumed to be separated and of finite type over the $p$-adic number field $\mathbb{Q}_p$. Assume that the corresponding category $\mathbb{D}_{\mathrm{perf}}(X_\sharp,T)$ for $\sharp=\text{\'et}$ admits Waldhausen structure and we have the corresponding Waldhausen exact functor (induced from the corresponding direct image functor as in \cite[Chapter 4-5]{Wit3}) $R\Gamma_?(X_\sharp,.)$ where $?=\emptyset$:
 	\[
\xymatrix@R+0pc@C+3pc{
\mathbb{D}_{\mathrm{perf}}(X_\sharp,T)\ar[r]^{R\Gamma_?(X_\sharp,.)}\ar[r]\ar[r] &\mathbb{D}_{\mathrm{perf}}(T)
}
\]
which induces the corresponding the morphism $\mathbb{K}R\Gamma_?(X_\sharp,.)$ as:
\[
\xymatrix@R+0pc@C+3pc{
\mathbb{K}\mathbb{D}_{\mathrm{perf}}(X_\sharp,T)\ar[r]^{\mathbb{K}R\Gamma_?(X_\sharp,.)}\ar[r]\ar[r] & \mathbb{K}\mathbb{D}_{\mathrm{perf}}(T).
}
\]	
Then we have that this map is homotopic to the zero map:
\[
\xymatrix@R+0pc@C+3pc{
\mathbb{K}\mathbb{D}_{\mathrm{perf}}(X_\sharp,T)\ar[r]^{0}\ar[r]\ar[r] & \mathbb{K}\mathbb{D}_{\mathrm{perf}}(T).
}
\]

\end{conjecture}

\begin{remark}	

The proof should be actually the same as in the previous result in the corresponding $\ell$-adic situation, the difference here is that one uses the corresponding finiteness and cohomological dimension results in the $p$-adic setting by applying \cite[8.1,10.1]{KL} for instance, which is to say in our setting the corresponding finiteness and cohomological dimension results over the point attached to $\mathbb{Q}_p^\mathrm{ur}$ would produce us the corresponding cofibration sequence involved. Also we need to consider a general ring $T$ here. Following the idea in \cite[Proposition 2.1.3]{FK}, we quotient out the corresponding Jacobson radical of each finite quotient $T/I$ which is a finite ring. We can then reduce to semi-simple case, and then reduce to finite field case. 
\end{remark}

\begin{conjecture}
When $T$ is as in the previous proposition the statement in the previous corollary holds as well in our current situation, which is to say over $\mathrm{Spa}(\mathbb{Q}_p,\mathfrak{o}_{\mathbb{Q}_p})$ we have the Waldhausen category $\mathbb{D}_{\mathrm{perf}}(X_\sharp,T)$ for $\sharp=\text{\'et}$ with the corresponding Waldhausen exact functor $R\Gamma_?(X_\sharp,.)$ for $?=\emptyset$:
\[
\xymatrix@R+0pc@C+3pc{
\mathbb{D}_{\mathrm{perf}}(X_\sharp,T)\ar[r]^{R\Gamma_?(X_\sharp,.)}\ar[r]\ar[r] &\mathbb{D}_{\mathrm{perf}}(T)
}
\]
which induces the corresponding the morphism $\mathbb{K}R\Gamma_?(X_\sharp,.)$ as:
\[
\xymatrix@R+0pc@C+3pc{
\mathbb{K}\mathbb{D}_{\mathrm{perf}}(X_\sharp,T)\ar[r]^{\mathbb{K}R\Gamma_?(X_\sharp,.)}\ar[r]\ar[r] & \mathbb{K}\mathbb{D}_{\mathrm{perf}}(T).
}
\]	
And we have that this map is homotopic to the zero map:
\[
\xymatrix@R+0pc@C+3pc{
\mathbb{K}\mathbb{D}_{\mathrm{perf}}(X_\sharp,T)\ar[r]^{0}\ar[r]\ar[r] & \mathbb{K}\mathbb{D}_{\mathrm{perf}}(T).
}
\]
\end{conjecture}


\begin{remark}
We used the corresponding context of Waldhausen categories after Witte, but certain cases of the corresponding $\epsilon$-isomorphism conjectures in the \'etale setting are due to many people who studied the corresponding $\epsilon$-isomorphism conjectures, mainly in \cite{FK}, \cite{F},\cite{BB}, \cite{LZ} and \cite{Nak3}.	
\end{remark}

\newpage

\subsection*{Acknowledgements} 
We would like to thank the philosophy proposed in \cite{KP1} and the paper itself which inspired us to initiate our project on the unification and the generalization of $p$-adic Hodge theory and the corresponding Iwasawa theory in the deep philosophy of Fukaya-Kato in the noncommutative setting in both Deligne's category of virtual objects and Waldhausen category (even more sophisticated $\infty$-categorical consideration). For the Iwasawa theory part, we are actually deeply inspired by not only the Burns-Flach-Fukaya-Kato-Nakamura's Iwasawa theoretic philosophy but also some deep $K$-theoretic point of view coming from Witte's geometric Iwasawa theory for constructible $\ell$-adic sheaves over a scheme. The author would like to thank Professor Kiran Kedlaya for helpful discussion, which deepens our understanding on \cite{KL15} and \cite{KL16} in some quite nontrivial way.

\newpage

\bibliographystyle{ams}

\end{document}